\documentclass{article}%
\usepackage{amssymb}
\usepackage{amsmath}%
\usepackage{amsfonts}%
\usepackage{graphicx}
\providecommand{\U}[1]{\protect\rule{.1in}{.1in}}
\newtheorem{theorem}{Theorem}

\newtheorem{corollary}[theorem]{Corollary}

\newtheorem{lemma}[theorem]{Lemma}

\newenvironment{proof}[1][Proof]{\noindent\textbf{#1.} }{\ \rule{0.5em}{0.5em}}
\begin{document}

\title{A Spectral Method\\for Nonlinear Elliptic Equations}
\author{Kendall Atkinson\\Departments of Mathematics \& Computer Science \\The University of Iowa
\and David Chien\\Department of Mathematics \\California State University - San Marcos
\and Olaf Hansen\\Department of Mathematics \\California State University - San Marcos}
\maketitle

\begin{abstract}
Let $\Omega$ be an open, simply connected, and bounded region in
$\mathbb{R}^{d}$, $d\geq2$, and assume its boundary $\partial\Omega$ is
smooth. Consider solving an elliptic partial differential equation $Lu=f$ over
$\Omega$ with zero Dirichlet boundary value. The function $f$ is a nonlinear
function of the solution $u$. The problem is converted to an
equivalent\ elliptic problem over the open unit ball $\mathbb{B}^{d}$ in
$\mathbb{R}^{d}$, say $\widetilde{L}\widetilde{u}=\widetilde{f}$. Then a
spectral Galerkin method is used to create a convergent sequence of
multivariate polynomials $\widetilde{u}_{n}$ of degree $\leq n$ that is
convergent to $\widetilde{u}$. The transformation from $\Omega$ to
$\mathbb{B}^{d}$ requires a special analytical calculation for its
implementation. With sufficiently smooth problem parameters, the method is
shown to be rapidly convergent. For $u\in C^{\infty}\left(  \overline{\Omega
}\right)  $ and assuming $\partial\Omega$ is a $C^{\infty}$ boundary, the
convergence of $\left\Vert \widetilde{u}-\widetilde{u}_{n}\right\Vert _{H^{1}%
}$ \ to zero is faster than any power of $1/n$. Numerical examples illustrate
experimentally an exponential rate of convergence. A generalization to
$-\Delta u+\gamma u=f$ with a zero Neumann boundary condition is also presented.

\end{abstract}

\section{Introduction}

Consider the nonlinear problem
\begin{equation}
Lu\left(  s\right)  =f\left(  s,u(s)\right)  ,\quad\quad s\in\Omega\label{en1}%
\end{equation}%
\begin{equation}
u\left(  s\right)  =0,\quad\quad s\in\mathbb{\partial}\Omega\label{en1a}%
\end{equation}
with $L$ an elliptic operator over $\Omega$ and a Dirichlet boundary
condition. Let $\Omega$ be an open, simply--connected, and bounded region in
$\mathbb{R}^{d}$, and assume that its boundary $\partial\Omega$ is smooth and
sufficiently differentiable. Assume $L$ is a strongly elliptic operator of the
form%
\[
Lu(s)\equiv-\sum_{i,j=1}^{d}\frac{\partial}{\partial s_{i}}\left(
a_{i,j}(s)\frac{\partial u(s)}{\partial s_{j}}\right)  +\gamma\left(
s\right)  u\left(  s\right)  ,\quad\quad s\in\Omega,
\]
The functions $a_{i,j}(s)$, $1\leq i,j\leq d$, are assumed to be several times
continuously differentiable over $\overline{\Omega}$, and the $d\times d$
\ matrix $\left[  a_{i,j}\left(  s\right)  \right]  $ is to be symmetric and
to satisfy%
\begin{equation}
\xi^{\text{T}}A(s)\xi\geq\alpha\xi^{\text{T}}\xi,\quad\quad s\in
\overline{\Omega},\quad\xi\in\mathbb{R}^{d}\label{en2}%
\end{equation}
for some $\alpha>-\infty$. We also assume the coefficient $\gamma\in C\left(
\overline{\Omega}\right)  $. Note that because the right-hand function $f$ is
allowed to depend on $u$, we can add to each side of (\ref{en1}) an
arbitrarily large multiple of $u$.

The problem (\ref{en1})-(\ref{en1a}) can be reformulated as a variational
problem. Introduce%
\begin{equation}%
\begin{array}
[c]{r}%
\mathcal{A}\left(  v,w\right)  =%
{\displaystyle\int_{\mathbb{B}^{d}}}
\left[
{\displaystyle\sum\limits_{i,j=1}^{d}}
a_{i,j}(s)\dfrac{\partial v(s)}{\partial s_{i}}\dfrac{\partial w(s)}{\partial
s_{j}}\right]  \,ds\quad\quad\quad\\
+%
{\displaystyle\int_{\mathbb{B}^{d}}}
\gamma\left(  s\right)  v\left(  s\right)  w\left(  s\right)  \,ds\quad\quad
v,w\in H_{0}^{1}\left(  \Omega\right)  .
\end{array}
\label{en2a}%
\end{equation}%
\begin{equation}
\left(  \mathcal{F}\left(  v\right)  \right)  \left(  s\right)  =f\left(
s,v(s)\right)  ,\quad\quad s\in\Omega,\quad\quad v\in H^{1}\left(
\Omega\right)  .\label{en2b}%
\end{equation}
We note that the Sobolev space $H^{m}\left(  \Omega\right)  $ is the closure
of $C^{m}\left(  \overline{\Omega}\right)  $ using the norm%
\[
\left\Vert g\right\Vert _{H^{m}\left(  \Omega\right)  }=\sqrt{\sum_{\left\vert
i\right\vert \leq m}\left\Vert D^{i}g\right\Vert _{L_{2}\left(  \Omega\right)
}^{2}},\quad\quad f\in C^{m}\left(  \overline{\Omega}\right)  ,\quad m\geq1
\]
with $i$ a multi-integer, $i=\left(  i_{1},\dots,i_{d}\right)  ,$ $\left\vert
i\right\vert =i_{1}+\cdots+i_{d}$, and%
\[
D^{i}g\left(  s\right)  =\frac{\partial^{\left\vert i\right\vert }g\left(
s\right)  }{\partial s_{1}^{i_{1}}\cdots\partial s_{d}^{i_{d}}}.
\]
The space $H_{0}^{1}\left(  \Omega\right)  $ is the closure of $C_{0}%
^{1}\left(  \Omega\right)  $ using $\left\Vert \cdot\right\Vert _{H^{1}\left(
\Omega\right)  }$, where elements \ of $C_{0}^{1}\left(  \Omega\right)
\subseteq$ $C^{1}\left(  \overline{\Omega}\right)  $are zero on some open
neighborhood of the boundary of $\Omega$.

Noting (\ref{en2}) and choosing a sufficiently large positive value for%
\[
\min_{s\in\overline{\Omega}}\gamma\left(  s\right)
\]
(say by adding a sufficiently large multiple of $u$ to both sides of
(\ref{en1})), we can assume that $\mathcal{A}$ is a strongly elliptic operator
on $H_{0}^{1}\left(  \Omega\right)  $, namely%
\[
\mathcal{A}\left( v,v\right)  \geq c_{0}\left\Vert v\right\Vert _{H^{1}\left(
\Omega\right)  }^{2} ,\quad\quad\forall v\in H^{1}\left(  \Omega\right)
\]
for some finite $c_{0}>0$.

Reformulate (\ref{en1})-(\ref{en1a}) as the following variational problem:
find $u\in H_{0}^{1}\left(  \Omega\right)  $ for which%
\begin{equation}
\mathcal{A}\left(  u,w\right)  =\left(  \mathcal{F}\left(  u\right)
,w\right)  ,\quad\quad\forall w\in H_{0}^{1}\left(  \Omega\right)
.\label{en3}%
\end{equation}
Throughout this paper we assume the variational reformulation of the problem
(\ref{en1})-(\ref{en1a}) has a locally unique solution $u\in H_{0}^{1}\left(
\Omega\right)  $. For analyses of the existence and uniqueness of a solution
to (\ref{en1})-(\ref{en1a}), see Zeidler \cite{zeidler}.

In the following \S \ref{NumMethod} we define our spectral method for the case
that $\Omega=\mathbb{B}^{d}$; and following that we show how to reformulate
the problem (\ref{en1})-(\ref{en1a}) for a general region $\Omega$ as an
equivalent problem over $\mathbb{B}^{d}$. This follows the earlier development
in \cite{ach2008}. In \S \ref{ConvAnal} we present a convergence analysis for
our numerical method. Implementation of the method is discussed in
\S \ref{implement}, followed by numerical examples in \S \ref{NumExam}. An
extension to a Neumann boundary condition is given in \S \ref{Nmn}.

\section{A spectral method on the unit ball\label{NumMethod}}

Let $\mathcal{X}_{n}$ denote a finite-dimensional subspace of $H_{0}%
^{1}\left(  \mathbb{B}^{d}\right)  $, and let $\left\{  \psi_{1},\dots
,\psi_{N_{n}}\right\}  $ be a basis of $\mathcal{X}_{n}$. Later we define such
a basis by using polynomials of degree $\leq n$ over $\mathbb{R}^{d}$, denoted
by $\Pi_{n}^{d}$, and $N_{n}$ is the dimension of $\Pi_{n}^{d}$. We seek an
approximating solution to (\ref{en3}) by finding $u_{n}\in\mathcal{X}_{n}$
such that%
\begin{equation}
\mathcal{A}\left(  u_{n},w\right)  =\left(  \mathcal{F}\left(  u_{n}\right)
,w\right)  ,\quad\quad\forall w\in\mathcal{X}_{n}.\label{en4}%
\end{equation}
More precisely, find
\begin{equation}
u_{n}\left(  x\right)  =\sum_{\ell=1}^{N_{n}}\alpha_{\ell}\psi_{\ell}\left(
x\right) \label{en5}%
\end{equation}
that satisfies the nonlinear algebraic system%
\begin{equation}%
\begin{array}
[c]{l}%
{\displaystyle\sum\limits_{\ell=1}^{N_{n}}}
\alpha_{\ell}%
{\displaystyle\int_{\mathbb{B}^{d}}}
\left[
{\displaystyle\sum\limits_{i,j=1}^{d}}
a_{i,j}(x)\dfrac{\partial\psi_{\ell}(x)}{\partial x_{i}}\dfrac{\partial
\psi_{k}(x)}{\partial x_{j}}+\gamma\left(  x\right)  \psi_{\ell}\left(
x\right)  \psi_{k}\left(  x\right)  \right]  \,dx\quad\smallskip\\
\quad=%
{\displaystyle\int_{\mathbb{B}^{d}}}
f\left(  x,%
{\displaystyle\sum\limits_{\ell=1}^{N_{n}}}
\alpha_{\ell}\psi_{\ell}\left(  x\right)  \right)  \psi_{k}(x)\,dx,\quad\quad
k=1,\dots,N_{n}.
\end{array}
\label{en6}%
\end{equation}
For notation, we generally use the variable $x$ when considering
$\mathbb{B}^{d}$, and we use the variable $s$ when considering $\Omega$.

To obtain a space for approximating the solution $u$ of our problem, we
proceed as follows. Denote by $\Pi_{n}^{d}$ the space of polynomials in $d$
variables that are of degree $\leq n$: $p\in\Pi_{n}^{d}$ if it has the form%
\[
p(x)=\sum_{\left\vert i\right\vert \leq n}a_{i}x_{1}^{i_{1}}x_{2}^{i_{2}}\dots
x_{d}^{i_{d}}.
\]
Our approximation space with respect to $\mathbb{B}^{d}$ is%
\begin{equation}
\mathcal{X}_{n}=\left\{  \left(  1-\left\vert x\right\vert ^{2}\right)
p(x)\mid p\in\Pi_{n}^{d}\right\}  \subseteq H_{0}^{1}\left(  \mathbb{B}%
^{d}\right) \label{e9a}%
\end{equation}
Let $N_{n}=\dim\mathcal{X}_{n}=\dim\Pi_{n}^{d}$. For $d=2$, $N_{n}=\left(
n+1\right)  \left(  n+2\right)  /2$. Practical implementation of the numerical
method (\ref{en4})-(\ref{en6}) is discussed in \S \ref{implement}.

\subsection{Transformation of the domain $\Omega$ \label{section_transform}}

For the more general problem (\ref{en1})-(\ref{en1a}) over a general region
$\Omega$, we reformulate it as a problem over $\mathbb{B}^{d}$. We review here
some \ ideas from \cite{ach2008}, referring the reader to it for additional details.

Assume the existence of a function%
\begin{equation}
\Phi:\overline{\mathbb{B}}^{d}\underset{onto}{\overset{1-1}{\longrightarrow}%
}\overline{\Omega}\label{e4}%
\end{equation}
with $\Phi$ a twice--differentiable mapping, and let $\Psi=\Phi^{-1}%
:\overline{\Omega}\underset{onto}{\overset{1-1}{\longrightarrow}}%
\overline{\mathbb{B}}^{d}$. \ For $v\in L^{2}\left(  \Omega\right)  $, let%
\begin{equation}
\widetilde{v}(x)=v\left(  \Phi\left(  x\right)  \right)  ,\quad\quad
x\in\overline{\mathbb{B}}^{d}\subseteq\mathbb{R}^{d}\label{e6}%
\end{equation}
and conversely,%
\begin{equation}
v(s)=\widetilde{v}\left(  \Psi\left(  s\right)  \right)  ,\quad\quad
s\in\overline{\Omega}\subseteq\mathbb{R}^{d}.\label{e7}%
\end{equation}
Assuming $v\in H^{1}\left(  \Omega\right)  $, we can show%
\[
\nabla_{x}\widetilde{v}\left(  x\right)  =J\left(  x\right)  ^{\text{T}}%
\nabla_{s}v\left(  s\right)  ,\quad\quad s=\Phi\left(  x\right)
\]
with $J\left(  x\right)  $ the Jacobian matrix for $\Phi$ over the unit ball
$\mathbb{B}^{d}$,%
\begin{equation}
J(x)\equiv\left(  D\Phi\right)  (x)=\left[  \frac{\partial\Phi_{i}%
(x)}{\partial x_{j}}\right]  _{i,j=1}^{d},\quad\quad x\in\overline{\mathbb{B}%
}^{d}.\label{e7b}%
\end{equation}
To use our method for problems over a region $\Omega$, it is necessary to know
explicitly the functions $\Phi$ and $J$. The creation of such a mapping $\Phi$
is taken up in \cite{ah2011} for cases in which only a boundary mapping is
known, from $\mathbb{S}^{d-1}$ to $\partial\Omega$, a common way to define the
region $\Omega$.

We assume%
\begin{equation}
\det J(x)\neq0,\quad\quad x\in\overline{\mathbb{B}}^{d}.\label{e8}%
\end{equation}
Similarly,%
\[
\nabla_{s}v(s)=K(s)^{\text{T}}\nabla_{x}\widetilde{v}(x),\quad\quad x=\Psi(s)
\]
with $K(s)$ the Jacobian matrix for $\Psi$ over $\Omega$. By differentiating
the identity
\[
\Psi\left(  \Phi\left(  x\right)  \right)  =x,\quad\quad x\in\overline
{\mathbb{B}}^{d}%
\]
we obtain%
\[
K\left(  \Phi\left(  x\right)  \right)  =J\left(  x\right)  ^{-1}.
\]
Assumptions about the differentiability of $\widetilde{v}\left(  x\right)  $
can be related back to assumptions on the differentiability of $v(s)$ and
$\Phi(x)$.

\begin{lemma}
\label{transform_lemma}Let $\Phi\in C^{m}\left(  \overline{\mathbb{B}}%
^{d}\right)  $. If $v\in C^{k}\left(  \overline{\Omega}\right)  $, then
$\widetilde{v}\in C^{q}\left(  \overline{\mathbb{B}}^{d}\right)  $ with
$q=\min\left\{  k,m\right\}  $. Similarly, if $v\in H^{k}\left(
\overline{\Omega}\right)  $, then $\widetilde{v}\in H^{q}\left(
\overline{\mathbb{B}}^{d}\right)  .$
\end{lemma}

\noindent A proof is straightforward using (\ref{e6}). A converse statement
can be made as regards $\widetilde{v}$, $v$, and $\Psi$ in (\ref{e7}).
Moreover, the differentiability of $\Phi$ over $\mathbb{B}^{d}$ is exactly the
same as that of $\Psi$ over $\Omega$.

Applying this transformation to the equation (\ref{en1}), we obtain%
\begin{equation}%
\begin{array}
[c]{r}%
-%
{\displaystyle\sum\limits_{i,j=1}^{d}}
\dfrac{\partial}{\partial x_{i}}\left(  \det\left(  J(x)\right)
\widetilde{a}_{i,j}(x)\dfrac{\partial\widetilde{u}(x)}{\partial x_{j}}\right)
+\widetilde{\gamma}\left(  x\right)  \widetilde{u}(x)\smallskip\\
=\widetilde{f}\left(  x,\widetilde{u}(x)\right)  ,\quad\quad x\in
\mathbb{B}^{d}%
\end{array}
\label{en8a}%
\end{equation}%
\begin{align}
\widetilde{f}\left(  x,\widetilde{u}(x)\right)   &  =\det\left(  J(x)\right)
f\left(  \Phi\left(  x\right)  ,\widetilde{u}(x)\right)  ,\quad\quad
x\in\mathbb{B}^{d}\label{en8aa}\\
\widetilde{\gamma}\left(  x\right)   &  =\det\left(  J(x)\right)
\gamma\left(  \Phi\left(  x\right)  \right) \label{en8ab}%
\end{align}%
\begin{align}
\widetilde{A}\left(  x\right)   &  =J\left(  x\right)  ^{-1}A(\Phi\left(
x\right)  )J\left(  x\right)  ^{-\text{T}}\nonumber\\
&  \equiv\left[  \widetilde{a}_{i,j}(x)\right]  _{i,j=1}^{d}\label{en8ac}%
\end{align}
A derivation of this is given in \cite[Thm. 3]{ach2008}. With (\ref{en8a}), we
also impose the Dirichlet condition
\begin{equation}
\widetilde{u}(x)=0,\quad\quad x\in\mathbb{B}^{d}\label{en8b}%
\end{equation}

The problem of solving (\ref{en8a}), (\ref{en8b}) is completely equivalent to
that of solving (\ref{en1}), (\ref{en1a}). Also, the differential operator in
(\ref{en8a}) will be strongly elliptic. As noted earlier, the creation of such
a mapping $\Phi$ is discussed at length in \cite{ah2011} for extending a
boundary mapping $\varphi:\mathbb{S}^{d-1}\rightarrow\partial\Omega$ to a
mapping $\Phi$ satisfying (\ref{e4}) and (\ref{e8}).

\section{Error analysis\label{ConvAnal}}

\textbf{\ }In \cite{Osborn1975} Osborn converted a finite element method for
solving an eigenvalue problem for an elliptic partial differential equation to
a corresponding numerical method for approximating the eigenvalues of a
compact integral operator. He then used results for the latter to obtain
convergence results for his finite element method. We use his construction to
convert the numerical method for (\ref{en3}) to a corresponding method for
finding a fixed point for a completely continuous nonlinear integral operator,
and this latter numerical method will be analyzed using the results given in
\cite[Chap. 3]{Kras1964} and \cite{atkin1973}.

Important results about polynomial approximation \ have been given recently by
Li and Xu \cite{LiXu}, and they are critical to our convergence analysis.

\begin{theorem}
\label{Thm1a}(Li and Xu) Let $r\geq2$. \textit{Given }$v\in H^{r}\left(
\mathbb{B}^{d}\right)  $\textit{, there exists a sequence of polynomials
}$p_{n}\in\Pi_{n}^{d}$ such that%
\begin{equation}
\left\Vert v-p_{n}\right\Vert _{H^{1}\left(  \mathbb{B}^{d}\right)  }%
\leq\varepsilon_{n,r}\left\Vert v\right\Vert _{H^{r}\left(  \mathbb{B}%
^{d}\right)  },\quad\quad n\geq1.\label{en9c}%
\end{equation}
\textit{The sequence }$\varepsilon_{n,r}=\mathcal{O}\left(  n^{-r+1}\right)
$\textit{\ and is independent of }$v$\textit{.\medskip}
\end{theorem}

\begin{theorem}
\label{Thm1b}(Li and Xu) Let $r\geq2$. \textit{Given }$v\in H_{0}^{1}\left(
\mathbb{B}^{d}\right)  \cap H^{r}\left(  \mathbb{B}^{d}\right)  $\textit{,
there exists a sequence of polynomials }$p_{n}\in\mathcal{X}_{n}%
$\textit{\ such that}%
\begin{equation}
\left\Vert v-p_{n}\right\Vert _{H^{1}\left(  \mathbb{B}^{d}\right)  }%
\leq\varepsilon_{n,r}\left\Vert v\right\Vert _{H^{r}\left(  \mathbb{B}%
^{d}\right)  },\quad\quad n\geq1.\label{en9d}%
\end{equation}
\textit{The sequence }$\varepsilon_{n,r}=\mathcal{O}\left(  n^{-r+1}\right)
$\textit{\ and is independent of }$v$\textit{.\medskip}
\end{theorem}

\noindent These two results are Theorems 4.2 and 4.3, respectively, in
\cite{LiXu}. For the second theorem, also see the comments immediately
following \cite[Thm. 4.3]{LiXu}.\medskip

For the convergence analysis, we follow closely the development in Osborn
\cite[\S 4(a)]{Osborn1975}. We omit the details, noting only those different
from \cite[\S 4(a)]{Osborn1975}. Taking $f$ to be a given function in
$L_{2}\left(  \mathbb{B}^{d}\right)  $, the solution of (\ref{en3}) can be
written as $u=\mathcal{T}f$ with $\mathcal{T}:L_{2}\left(  \mathbb{B}%
^{d}\right)  \rightarrow H_{0}^{1}\left(  \mathbb{B}^{d}\right)  \cap
H^{2}\left(  \mathbb{B}^{d}\right)  $ and bounded,%
\[
\left\Vert \mathcal{T}f\right\Vert _{H^{2}\left(  \mathbb{B}^{d}\right)  }\leq
C\,\left\Vert f\right\Vert _{L_{2}\left(  \mathbb{B}^{d}\right)  },\quad\quad
f\in L_{2}\left(  \mathbb{B}^{d}\right)  .
\]
The operator is the `Green's integral operator' for the associated Dirichlet
problem. More generally, for $r\geq0$, $\mathcal{T}:H^{r}\left(
\mathbb{B}^{d}\right)  \rightarrow H_{0}^{1}\left(  \mathbb{B}^{d}\right)
\cap H^{r+2}\left(  \mathbb{B}^{d}\right)  $,
\[
\left\Vert \mathcal{T}f\right\Vert _{H^{r+2}\left(  \mathbb{B}^{d}\right)
}\leq C_{r}\,\left\Vert f\right\Vert _{H^{r}\left(  \mathbb{B}^{d}\right)
},\quad\quad f\in H^{r}\left(  \mathbb{B}^{d}\right)  .
\]
In addition, $\mathcal{T}$ is a compact operator on $L_{2}\left(
\mathbb{B}^{d}\right)  $ into $H_{0}^{1}\left(  \mathbb{B}^{d}\right)  $, and
more generally, it is compact from $H^{r}\left(  \mathbb{B}^{d}\right)  $ into
$H_{0}^{1}\left(  \mathbb{B}^{d}\right)  \cap H^{r+1}\left(  \mathbb{B}%
^{d}\right)  $. With our assumptions, $\mathcal{T}$ is self-adjoint on
$L_{2}\left(  \mathbb{B}^{d}\right)  $, although Osborn allows more general
non-symmetric operators $L$. The same argument is applied to the numerical
method \ (\ref{en4}) to obtain a solution $u_{n}=\mathcal{T}_{n}f$ with
$\mathcal{T}_{n}$ having properties similar to $\mathcal{T}$ \ and also having
finite rank with range in $\mathcal{X}_{n}$.

The major assumption of Osborn is that his finite element method satisfies an
approximation inequality (see \cite[(4.7)]{Osborn1975}), and the above
theorems of Li and Xu are the corresponding statements for our numerical
method. The argument in \cite[\S 4(a)]{Osborn1975} then shows
\begin{equation}
\left\Vert \mathcal{T}-\mathcal{T}_{n}\right\Vert _{L_{2}\rightarrow L_{2}%
}\leq\frac{c}{n^{2}}.\label{en9}%
\end{equation}

Our variational problems (\ref{en3}) and (\ref{en4}) can now be reformulated
as
\begin{align}
u  &  =\mathcal{TF}\left(  u\right)  ,\label{en10}\\
u_{n}  &  =\mathcal{T}_{n}\mathcal{F}\left(  u_{n}\right)  ,\label{en11}%
\end{align}
and we regard these as equations on some subset of $L_{2}\left(
\mathbb{B}^{d}\right)  $, dependent on the form of the function $f$ defining
$\mathcal{F}$. The operator $\mathcal{F}$ of (\ref{en2b}) is sometimes called
the Nemytskii operator; see \cite[Chap. 1, \S 2]{Kras1964} for its properties.
It is necessary to assume that $\mathcal{F}$ is defined and continuous over
some open subset $D\subseteq L_{2}\left(  \mathbb{B}^{d}\right)  $:%
\[%
\begin{array}
[c]{c}%
v\in D\implies f\left(  \cdot,v\right)  \in L_{2}\left(  \mathbb{B}%
^{d}\right)  ,\\
v_{n}\rightarrow v\text{ in }L_{2}\left(  \mathbb{B}^{d}\right)  \implies
f\left(  \cdot,v_{n}\right)  \rightarrow f\left(  \cdot,v\right)  \text{ in
}L_{2}\left(  \mathbb{B}^{d}\right)  .
\end{array}
\]

The operators $\mathcal{T}$ and $\mathcal{T}_{n}$ are linear, and the
Nemytskii operator $\mathcal{F}$ provides the nonlinearity. The reformulation
(\ref{en10})-(\ref{en11}) can be used to give an error analysis of the
spectral method (\ref{en4}). The mapping $\mathcal{TF}$ is a compact nonlinear
operator on an open domain $D$ of a Banach space $\mathcal{X}$, in this case
$L_{2}\left(  \mathbb{B}^{d}\right)  $. Let $V\subseteq D$ be an open set
containing an isolated fixed point solution $u^{\ast}$ of (\ref{en10}). We can
define the index of $u^{\ast}$ (or more properly, the rotation of the vector
field $v-\mathcal{TF}\left(  v\right)  $ as $v$ varies over the boundary of
$V$); see \cite[Part II]{Kras1964}. For some intuition as to stability
implications of a fixed point having a nonzero index, see \cite[Property P5,
p. 802]{atkin1973}.

\begin{theorem}
\label{thm2}Assume the problem (\ref{en3}) has a solution $u^{\ast}$ that is
unique within some open neighborhood $V$ of $u^{\ast}$; further assume that
$u^{\ast}$ has nonzero index. Then for all sufficiently large $n$, (\ref{en4})
has one or more solutions $u_{n}$ within $V$, and all such $u_{n}$ converge to
$u^{\ast}$.
\end{theorem}

\begin{proof}
This is an application of the methods of \cite[Chap. 3, Sec. 3]{Kras1964} or
\cite[Thm. 3]{atkin1973}. A sufficient requirement is the norm convergence of
$\mathcal{T}_{n}$ to $\mathcal{T}$, given in (\ref{en9}); \cite[Thm.
3]{atkin1973} uses a weaker form of (\ref{en9}). $\left.  {}\right.  $%
\hfill\medskip
\end{proof}

The most standard case of a nonzero index involves a consideration of the
Frechet derivative of $\mathcal{F}$; see \cite[\S 5.3]{atkinson-han}. In
particular, the linear operator$\mathcal{F}^{\prime}\left(  v\right)  $ is
given by
\[
\left(  \mathcal{F}^{\prime}\left(  v\right)  w\right)  \left(  x\right)
=\left.  \frac{\partial f\left(  x,z\right)  }{\partial z}\right\vert
_{z=v(x)}\times w(x)
\]

\begin{theorem}
\label{ThmDir}Assume the problem (\ref{en3}) has a solution $u^{\ast}$ that is
unique within some open neighborhood $V$ of $u^{\ast}$; and further assume
that $I-\mathcal{TF}^{\prime}\left(  u^{\ast}\right)  $ is invertible over
$L_{2}\left(  \Omega\right)  $. Then $u^{\ast}$ has a nonzero index. Moreover,
for all sufficiently large $n$ there is a unique solution $u_{n}^{\ast}$ to
(\ref{en11}) within $V$, and $u_{n}^{\ast}$ converges to $u^{\ast}$ with
\begin{align}
\left\Vert u^{\ast}-u_{n}^{\ast}\right\Vert _{L_{2}\left(  \mathbb{B}%
^{d}\right)  }  &  \leq c\left\Vert \left(  \mathcal{T}-\mathcal{T}%
_{n}\right)  \mathcal{F}\left(  u^{\ast}\right)  \right\Vert _{L_{2}\left(
\mathbb{B}^{d}\right)  }\nonumber\\
&  \leq\frac{c}{n^{2}}\left\Vert \mathcal{F}\left(  u^{\ast}\right)
\right\Vert _{L_{2}\left(  \mathbb{B}^{d}\right)  }\label{en11a}%
\end{align}

\end{theorem}

\begin{proof}
Again this is an immediate application of results in \cite[Chap. 3, Sec.
3]{Kras1964} or \cite[Thm. 4]{atkin1973}. $\left.  {}\right.  $\hfill\medskip
\end{proof}

To improve upon this last result, we need to bound $\left\Vert \left(
\mathcal{T}-\mathcal{T}_{n}\right)  g\right\Vert _{L_{2}\left(  \mathbb{B}%
^{d}\right)  }$ when $g\in H^{r}$ for some $r\geq1$. \ Adapting the proof of
\cite[(4.9)]{Osborn1975} to our polynomial approximations and using Theorem
\ref{Thm1b},
\[
\left\Vert \left(  \mathcal{T}-\mathcal{T}_{n}\right)  g\right\Vert
_{H^{1}\left(  \mathbb{B}^{d}\right)  }\leq\frac{c}{n^{r+1}}\left\Vert
g\right\Vert _{H^{r}\left(  \mathbb{B}^{d}\right)  .}%
\]
Using the conservative bound%
\[
\left\Vert v\right\Vert _{L_{2}\left(  \mathbb{B}^{d}\right)  }\leq\left\Vert
v\right\Vert _{H^{1}\left(  \mathbb{B}^{d}\right)  },
\]
we have
\begin{equation}
\left\Vert \left(  \mathcal{T}-\mathcal{T}_{n}\right)  g\right\Vert
_{L_{2}\left(  \mathbb{B}^{d}\right)  }\leq\frac{c}{n^{r+1}}\left\Vert
g\right\Vert _{H^{r}\left(  \mathbb{B}^{d}\right)  .}\label{en12}%
\end{equation}

\begin{corollary}
\label{CorDir}For some $r\geq0$, assume $\mathcal{F}\left(  u^{\ast}\right)
\in H^{r}\left(  \mathbb{B}^{d}\right)  $. Then%
\begin{equation}
\left\Vert u^{\ast}-u_{n}^{\ast}\right\Vert _{L_{2}\left(  \mathbb{B}%
^{d}\right)  }\leq\mathcal{O}\left(  n^{-(r+1)}\right)  \left\Vert
\mathcal{F}\left(  u^{\ast}\right)  \right\Vert _{H^{r}\left(  \mathbb{B}%
^{d}\right)  }.\label{en13}%
\end{equation}

\end{corollary}

We conjecture that this bound and (\ref{en12}) can be improved to
$\mathcal{O}\left(  n^{-(r+2)}\right)  $. For the case $r=0$, an improved
result is given \ by (\ref{en11a}).

\subsection{A nonhomogeneous boundary condition}

Consider replacing the homogeneous boundary condition (\ref{en1a}) with the
nonhomogeneous condition
\[
u\left(  s\right)  =g\left(  s\right)  ,\quad\quad s\in\mathbb{\partial}%
\Omega,
\]
in which $g$ is a continuously differentiable function over $\partial\Omega$.
One possible approach to solving the Dirichlet problem with this nonzero
boundary condition is to begin by calculating a differentiable extension of
$g$, call it $G:\overline{\Omega}\rightarrow\mathbb{R}$, with%
\begin{align*}
G  &  \in C^{2}\left(  \overline{\Omega}\right)  ,\\
G\left(  s\right)   &  =g\left(  s\right)  ,\quad\quad s\in\partial\Omega
\end{align*}
With such a function $G$, introduce $v=u-G$ where $u$ satisfies (\ref{en1}%
)-(\ref{en1a}). Then $v$ satisfies the equation%
\begin{equation}
Lv\left(  s\right)  =f\left(  s,v(s)+G(s)\right)  -LG\left(  s\right)
,\quad\quad s\in\Omega,\label{en13a}%
\end{equation}%
\begin{equation}
v\left(  s\right)  =0,\quad\quad s\in\mathbb{\partial}\Omega.\label{en13b}%
\end{equation}
This problem is in the format of (\ref{en1})-(\ref{en1a}).

Sometimes finding an extension \ $G$ is straightforward; for example,
$g\equiv1$ over $\partial\Omega$ has the obvious extension $G\left(  s\right)
\equiv1$. Often, however, we must compute an extension. We begin by first
obtaining an extension $G$ using a method from \cite{ah2011}, and then we
approximate it with a polynomial of some reasonably low degree. For example,
see the construction of least squares approximants in \cite{ach2013}.

\section{Implementation\label{implement}}

We consider how to set up the nonlinear system of (\ref{en4})-(\ref{en6}) and
how to solve it. Because we intend to apply the method to problems defined
initially over a region $\Omega$ other than $\mathbb{B}^{d}$, we re-write
(\ref{en4})-(\ref{en6}) for this situation. The transformed equation we are
considering is the equation (\ref{en8a}). We look for a solution
\[
\widetilde{u}_{n}\left(  x\right)  =\sum_{\ell=1}^{N_{n}}\alpha_{\ell}%
\psi_{\ell}\left(  x\right)  ,
\]
and $u_{n}\left(  s\right)  $ is to be the equivalent solution considered over
$\Omega$: $\widetilde{u}_{n}\left(  x\right)  \equiv u_{n}\left(  \Phi\left(
x\right)  \right)  $, $x\in\mathbb{B}^{d}$. The coefficients $\left\{
\alpha_{\ell}|\ell=1,2,\dots,N_{n}\right\}  $ are the solutions of
\begin{equation}%
\begin{array}
[c]{r}%
{\displaystyle\sum\limits_{k=1}^{N_{n}}}
\alpha_{k}%
{\displaystyle\int_{\mathbb{B}^{d}}}
\left[
{\displaystyle\sum\limits_{i,j=1}^{d}}
\det J\left(  x\right)  \,\widetilde{a}_{i,j}(x)\dfrac{\partial\psi_{k}%
(x)}{\partial x_{j}}\dfrac{\partial\psi_{\ell}(x)}{\partial x_{i}%
}+\widetilde{\gamma}(x)\psi_{k}(x)\psi_{\ell}(x)\right.  \smallskip\\
\left.  +\widetilde{\gamma}(x)\psi_{k}(x)\psi_{\ell}(x)\right]  \,dx\quad
\quad\quad\quad\smallskip\\
=%
{\displaystyle\int_{\mathbb{B}^{d}}}
\widetilde{f}\left(  x,%
{\displaystyle\sum\limits_{k=1}^{N_{n}}}
\alpha_{k}\psi_{k}\left(  x\right)  \right)  \psi_{\ell}\left(  x\right)
\,dx,\quad\quad\ell=1,\dots,N_{n}%
\end{array}
\label{e78}%
\end{equation}
For the definitions of $\widetilde{\gamma}$, $\widetilde{f}$, and
$\widetilde{A}\left(  x\right)  \equiv\left[  \widetilde{a}_{i,j}(x)\right]
_{i,j=1}^{d}$, recall (\ref{en8aa})-(\ref{en8ac}).

When solving the nonlinear system (\ref{e78}), it is necessary to have an
initial guess $\widetilde{u}_{n}^{(0)}\left(  x\right)  =\sum_{\ell=1}^{N_{n}%
}\alpha_{\ell}^{(0)}\psi_{\ell}\left(  x\right)  $. In our examples, we begin
with a very small value for $n$ (say $n=1$), use $\widetilde{u}_{n}^{(0)}=0$,
and then solve (\ref{e78}) by some iterative method. Then increase $n$, using
as an initial guess the final solution obtained with a preceding $n$. This has
worked well in our computations, allowing us to work our way to the solution
of (\ref{e78}) for much larger values of $n$. For the iterative solver, we
have used the \textsc{Matlab} program \texttt{fsolve}, but will work in the
future on improving it.

\subsection{Planar problems}

The dimension of $\Pi_{n}^{2}$ is
\[
N_{n}=\frac{1}{2}\left(  n+1\right)  \left(  n+2\right)
\]
For notation, we replace $x$ with $\left(  x,y\right)  $. We create a basis
for $\mathcal{X}_{n}$ by first choosing an orthonormal basis for $\Pi_{n}^{2}
$, say $\left\{  \varphi_{m,k}|k=0,1,\dots,m;\,m=0,1,\dots,n\right\}  $. Then
define%
\begin{equation}
\psi_{m,k}\left(  x,y\right)  =\left(  1-x^{2}-y^{2}\right)  \varphi
_{m,k}\left(  x,y\right)  .\label{en75}%
\end{equation}
How do we choose the orthonormal basis $\left\{  \varphi_{\ell}(x,y)\right\}
_{\ell=1}^{N}$ for $\Pi_{n}^{2}$? Unlike the situation for the single variable
case, there are many possible orthonormal bases over $\mathbb{B}^{2} $, the
unit disk in $\mathbb{R}^{2}$. We have chosen one that is convenient for our
computations. These are the "ridge polynomials" introduced by Logan and\ Shepp
\cite{Loga} for solving an image reconstruction problem. A choice that is more
efficient in calculational costs is given in \cite{ach2013}; but we continue
to use the ridge polynomials because we are re-using and modifying computer
code written previously for use in \cite{ach2008}, \cite{ach2013},
\cite{ahc2009}, and \cite{ahc2013}.

We summarize here the results needed for our work. For general $d\geq2$, Let
\[
\mathcal{V}_{n}=\left\{  P\in\Pi_{n}^{d}:\left(  P,Q\right)  =0\quad\forall
Q\in\Pi_{n-1}^{d}\right\}
\]
the polynomials of degree $n$ that are orthogonal to all elements of
$\Pi_{n-1}^{d}$. Then
\begin{equation}
\Pi_{n}^{d}=\mathcal{V}_{0}\oplus\mathcal{V}_{1}\oplus\cdots\oplus
\mathcal{V}_{n}\label{e100}%
\end{equation}
is a decomposition of $\Pi_{n}^{d}$ into orthonormal subspaces. \ It is
standard to construct orthonormal bases of each $\mathcal{V}_{n}$ and to then
combine them to form an orthonormal basis of $\Pi_{n}^{d}$ using this
decomposition. \ 

For $d=2$, $\mathcal{V}_{n}$ has dimension $n+1$, $n\geq0$. As an orthonormal
basis of $\mathcal{V}_{n}$ we use%
\begin{equation}
\varphi_{n,k}(x,y)=\frac{1}{\sqrt{\pi}}U_{n}\left(  x\cos\left(  kh\right)
+y\sin\left(  kh\right)  \right)  ,\quad\left(  x,y\right)  \in D,\quad
h=\frac{\pi}{n+1}\label{e101}%
\end{equation}
for $k=0,1,\dots,n$. The function $U_{n}$ is the Chebyshev polynomial of the
second kind of degree $n$:%
\[
U_{n}(t)=\frac{\sin\left(  n+1\right)  \theta}{\sin\theta},\quad\quad
t=\cos\theta,\quad-1\leq t\leq1,\quad n=0,1,\dots
\]
The family $\left\{  \varphi_{n,k}\right\}  _{k=0}^{n}$ is an orthonormal
basis of $\mathcal{V}_{n}$.

As a basis of $\Pi_{n}^{2}$, we order $\left\{  \varphi_{m,k}\right\}  $
lexicographically based on the ordering in (\ref{e101}) and (\ref{e100}):%
\[
\left\{  \varphi_{\ell}\right\}  _{\ell=1}^{N}=\left\{  \varphi_{0,0}%
,\,\varphi_{1,0},\,\varphi_{1,1},\,\varphi_{2,0},\,\dots,\,\varphi
_{n,0},\,\dots,\varphi_{n,n}\right\}
\]
From (\ref{en75}), the family $\left\{  \psi_{m,k}\right\}  $ is ordered the same.

To calculate the first order partial derivatives of $\psi_{n,k}(x,y)$, we need
$U_{n}^{^{\prime}}(t)$. The values of $U_{n}(t)$ and $U_{n}^{^{\prime}}(t)$
are evaluated using the standard triple recursion relations%
\begin{align*}
U_{n+1}(t)  &  =2tU_{n}(t)-U_{n-1}(t)\\
U_{n+1}^{^{\prime}}(t)  &  =2U_{n}(t)+2tU_{n}^{^{\prime}}(t)-U_{n-1}%
^{^{\prime}}(t)
\end{align*}

For the numerical approximation of the integrals in (\ref{e78}), which are
over $\mathbb{B}^{2}$, the unit disk, we use the formula%
\begin{equation}
\int_{\mathbb{B}^{2}}g(x,y)\,dx\,dy\approx\sum_{l=0}^{q}\sum_{m=0}%
^{2q}g\left(  r_{l},\frac{2\pi\,m}{2q+1}\right)  \omega_{l}\frac{2\pi}%
{2q+1}r_{l}\label{e106}%
\end{equation}
Here the numbers $r_{l}$ and $\omega_{l}$ are the nodes and weights of the
$\left(  q+1\right)  $-point Gauss-Legendre quadrature formula on $[0,1]$.
Note that
\[
\int_{0}^{1}p(x)dx=\sum_{l=0}^{q}p(r_{l})\omega_{l},
\]
for all single-variable polynomials $p(x)$ with $\deg\left(  p\right)
\leq2q+1 $. The formula (\ref{e106}) uses the trapezoidal rule with $2q+1$
subdivisions for the integration over $\mathbb{B}^{2}$ in the azimuthal
variable. This quadrature (\ref{e106}) is exact for all polynomials $g\in
\Pi_{2q}^{2}$.

\subsection{The three--dimensional case\label{3D case}}

We change our notation, replacing $x\in\mathbb{B}^{3}$ with $\left(
x,y,z\right)  $. In $\mathbb{R}^{3}$, the dimension of $\Pi_{n}^{3}$ is
\[
N_{n}=\binom{n+3}{3}=\frac{1}{6}\left(  n+1\right)  \left(  n+2\right)
\left(  n+3\right)  .
\]
Here we choose orthonormal polynomials on the unit ball as described in
\cite{DX},
\begin{align}
\varphi_{n,j,k}(x)  &  =\frac{1}{h_{n,j,k}}C_{n-j-k}^{j+k+\frac{3}{2}%
}(x)(1-x^{2})^{\frac{j}{2}}\times\nonumber\\
&  \makebox[1cm]{}C_{j}^{k+1}(\frac{y}{\sqrt{1-x^{2}}})(1-x^{2}-y^{2}%
)^{k/2}C_{k}^{\frac{1}{2}}(\frac{z}{\sqrt{1-x^{2}-y^{2}}})\label{e1001}\\
j,k  &  =0,\ldots,n,\quad j+k\leq n,\quad n\in\mathbb{N}\nonumber
\end{align}
The function $\varphi_{n,j,k}(x)$ is a polynomial of degree $n$, $h_{n,j,k}$
is a normalization constant, and the functions $C_{i}^{\lambda}$ are the
Gegenbauer polynomials. The orthonormal base $\{\varphi_{n,j,k}\}_{n,j,k}$ and
its properties can be found in \cite[Chapter 2]{DX}.

We can order the basis lexicographically. To calculate these polynomials we
use a three--term recursion whose coefficients are given in \cite{ach2013}.

For the numerical approximation of the integrals in (\ref{e78}), we use a
quadrature formula for the unit ball $\mathbb{B}^{3}$,
\begin{align*}
\int_{\mathbb{B}^{3}}g(x)\,dx  &  =\int_{0}^{1}\int_{0}^{2\pi}\int_{0}^{\pi
}\widetilde{g}(r,\theta,\phi)\,r^{2}\sin(\phi)\,d\phi\,d\theta\,dr\approx
Q_{q}[g],\medskip\\
Q_{q}[g]  &  :=\sum_{i=1}^{2q}\sum_{j=1}^{q}\sum_{k=1}^{q}\frac{\pi}%
{q}\,\omega_{j}\,\nu_{k}\widetilde{g}\left(  \frac{\zeta_{k}+1}{2},\frac
{\pi\,i}{2q},\arccos(\xi_{j})\right)  .
\end{align*}
Here $\widetilde{g}(r,\theta,\phi)=g(x)$ is the representation of $g$ in
spherical coordinates. For the $\theta$ integration we use the trapezoidal
rule, because the function is $2\pi-$periodic in $\theta$. For the $r$
direction we use the transformation
\begin{align*}
\int_{0}^{1}r^{2}v(r)\;dr  &  =\int_{-1}^{1}\left(  \frac{t+1}{2}\right)
^{2}v\left(  \frac{t+1}{2}\right)  \frac{dt}{2}\medskip\\
&  =\frac{1}{8}\int_{-1}^{1}(t+1)^{2}v\left(  \frac{t+1}{2}\right)
\;dt\medskip\\
&  \approx\sum_{k=1}^{q}\underset{_{=:\nu_{k}}}{\underbrace{\frac{1}{8}\nu
_{k}^{\prime}}}v\left(  \frac{\zeta_{k}+1}{2}\right)  ,
\end{align*}
where the $\nu_{k}^{\prime}$ and $\zeta_{k}$ are the weights and the nodes of
the Gauss quadrature with $q$ nodes on $[-1,1]$ with respect to the inner
product
\[
(v,w)=\int_{-1}^{1}(1+t)^{2}v(t)w(t)\,dt.
\]
The weights and nodes also depend on $q$ but we omit this index. For the
$\phi$ direction we use the transformation
\[
\int_{0}^{\pi}\sin(\phi)v(\phi)\,d\phi=\int_{-1}^{1}v(\arccos(\phi
))\,d\phi\approx\sum_{j=1}^{q}\omega_{j}v(\arccos(\xi_{j})),
\]
where the $\omega_{j}$ and $\xi_{j}$ are the nodes and weights for the
Gauss--Legendre quadrature on $[-1,1]$. For more information on this
quadrature rule on the unit ball in $\mathbb{R}^{3}$, see \cite{stroud}.

Finally we need the gradient to approximate the integral in (\ref{e78}). To do
this one can modify the three--term recursion in \cite{ach2013} to calculate
the partial derivatives of $\varphi_{n,j,k}(x)$.

\section{Numerical examples\label{NumExam}}

We begin with a planar example. Consider the problem%
\begin{equation}%
\begin{tabular}
[c]{cc}%
$-\Delta u\left(  s,t\right)  =f\left(  s,t,u\left(  s,t\right)  \right)  ,$ &
$\quad\left(  s,t\right)  \in\Omega$\\
$u\left(  s,t\right)  =0,$ & $\quad\left(  s,t\right)  \in\partial\Omega$%
\end{tabular}
\ \ \ \label{e120}%
\end{equation}
Note the change in notation, from $s\in\mathbb{R}^{2}$ to $\left(  s,t\right)
\in\mathbb{R}^{2}$.

As an illustrative region $\Omega$, we use the mapping $\Phi:\overline
{\mathbb{B}}^{2}\rightarrow\overline{\Omega}$, $\left(  s,t\right)
=\Phi\left(  x,y\right)  $,%
\begin{equation}%
\begin{array}
[c]{l}%
s=x-y+ax^{2}\\
t=x+y
\end{array}
\label{e132}%
\end{equation}
with $0<a<1$. It can be shown that $\Phi$ is a 1-1 mapping from the unit disk
$\overline{\mathbb{B}}^{2}$. In particular, the inverse mapping $\Psi
:\overline{\Omega}\rightarrow\overline{\mathbb{B}}^{2}$ is given by%
\begin{equation}%
\begin{array}
[c]{l}%
x=\dfrac{1}{a}\left[  -1+\sqrt{1+a\left(  s+t\right)  }\right]  \medskip\\
y=\dfrac{1}{a}\left[  at-\left(  -1+\sqrt{1+a\left(  s+t\right)  }\right)
\right]
\end{array}
\label{e133}%
\end{equation}

In Figure \ref{IntroOmega}(a), the mapping for $a=0.95$ is illustrated by
giving the images in $\overline{\Omega}$ of the circles $r=j/10$,
$j=1,\dots,10$ and the radial lines $\theta=j\pi/10$, $j=1,\dots,20$. An
alternative polynomial mapping $\Phi_{II}$ of degree 2 for this region is
computed using the integration/interpolation method of \cite[\S 3]{ah2011};
and $\Phi_{II}=\Phi$. on the boundary.$\partial\Omega$ as defined by
(\ref{e132}). It is illustrated in Figure \ref{IntroOmega}(b). This boundary
mapping $\Phi_{II}$ results in better error characteristics for our spectral
method as compared to the transformation $\Phi$.%

\begin{figure}[tbp] \centering
\begin{tabular}
[c]{ll}%
\raisebox{0.0104in}{\parbox[b]{2.0003in}{\begin{center}
\includegraphics[
height=2.0003in,
width=2.0003in
]%
{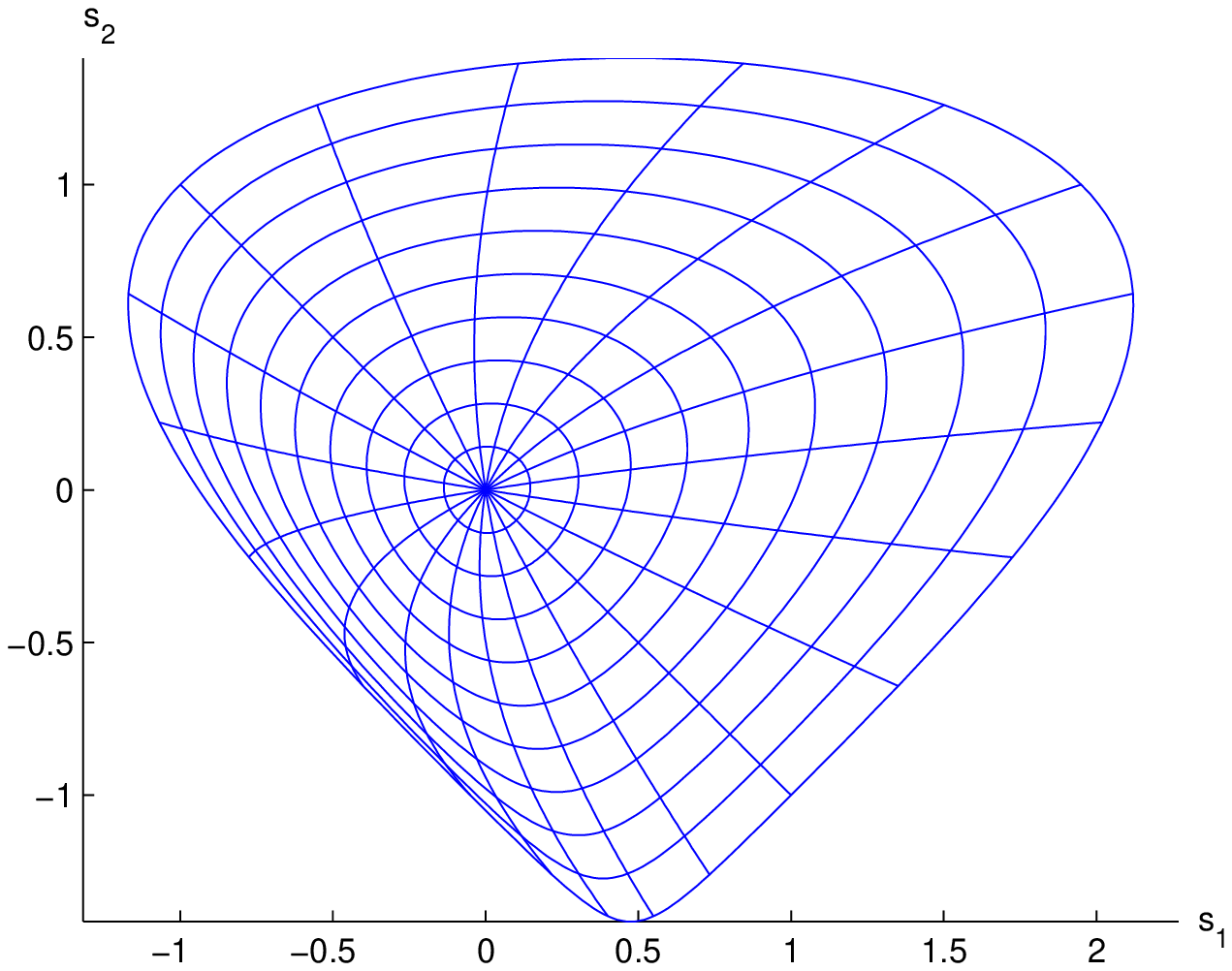}%
\\
(a) $\Phi$%
\end{center}}}
&
{\parbox[b]{2.0003in}{\begin{center}
\includegraphics[
height=2.0003in,
width=2.0003in
]%
{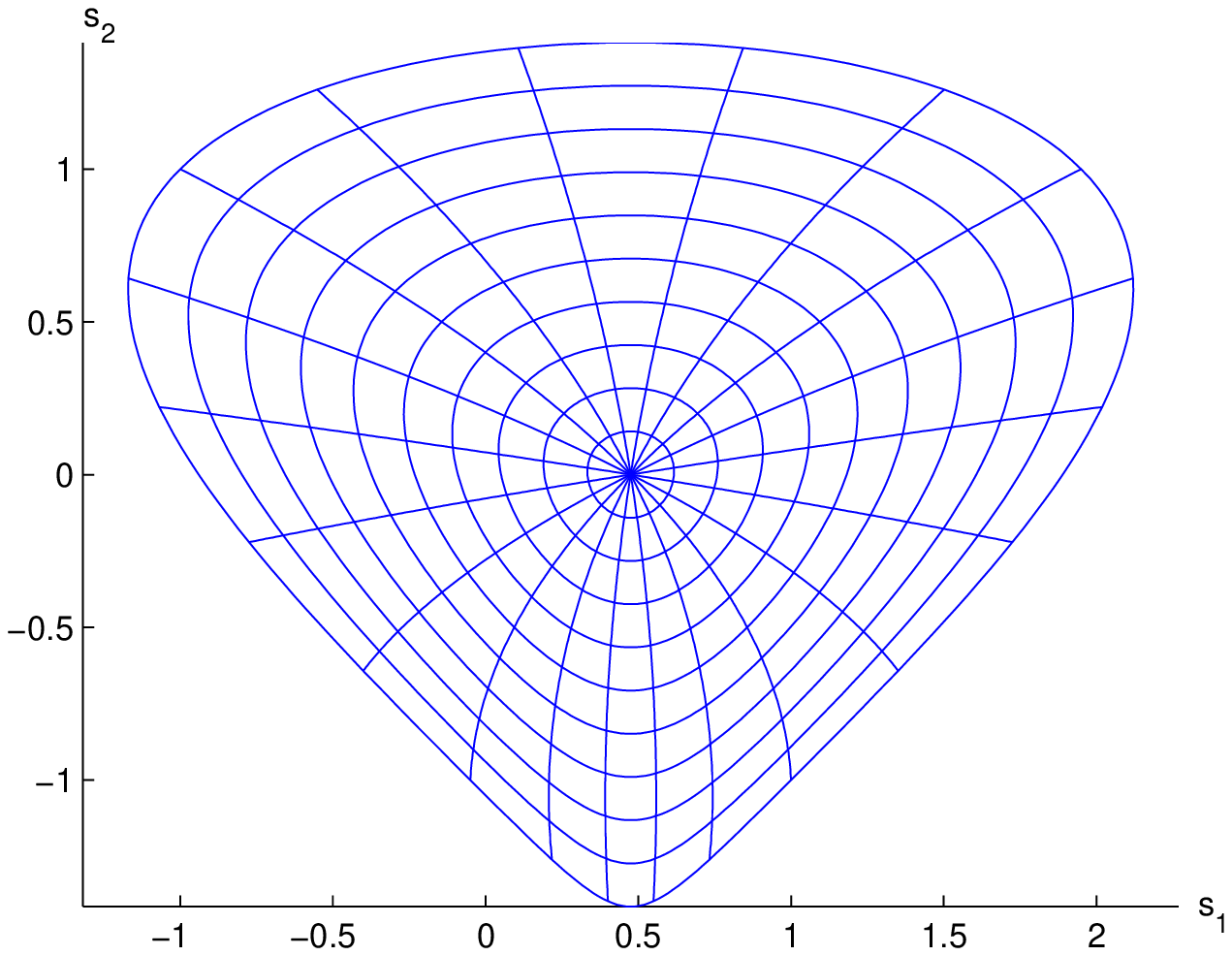}%
\\
(b) $\Phi_{II}$%
\end{center}}}
\end{tabular}
\caption{Illustrations of mappings on $\mathbb{B}^{2}$ for the region
$\Omega$ given by (\ref{e132})}\label{IntroOmega}%
\end{figure}%

As discussed earlier, we solve the nonlinear system (\ref{e78}) for a lower
value of the degree $n$, usually with an initial guess associated with
$u_{n}^{(0)}=0.$ As we increase $n$, we use the approximate solution from a
preceding $n$ to generate an initial guess for the new value of $n$. We use
the \textsc{Matlab} program \texttt{fsolve} to solve the nonlinear system. In
the future we plan to look at other numerical methods that take advantage of
the special structure of (\ref{e78}). \ To estimate the error, we use as a
true solution a numerical solution associated with a larger value of $n$.

For a particular case, consider the case%
\begin{equation}
f\left(  s,t,z\right)  =\frac{\cos\left(  \pi\,st\right)  }{1+z^{2}%
}\label{e135}%
\end{equation}
A graph of the solution is shown in Figure \ref{PlanarSoln}, along with
numerical results for $n=5,6,\dots,20$, with the solution $u_{25}$ taken as
the true solution. We use both the mapping $\Phi$ of (\ref{e132}) and the
mapping $\Phi_{II}$. \ Using either of the mappings, $\Phi$ or $\Phi_{II}$,
the graphs indicate an exponential rate of convergence for the mappings
$\left\{  u_{n}\right\}  $. The mapping $\Phi_{II}$ is better behaved, as can
be seen by visually comparing the graphs in \ref{IntroOmega}. This is the
probable reason for the improved convergence of the spectral method when using
$\Phi_{II}$ in comparison to $\Phi$.%

\begin{figure}[tbp] \centering
\begin{tabular}
[c]{ll}%
{\parbox[b]{2.0003in}{\begin{center}
\includegraphics[
height=2.0003in,
width=2.0003in
]%
{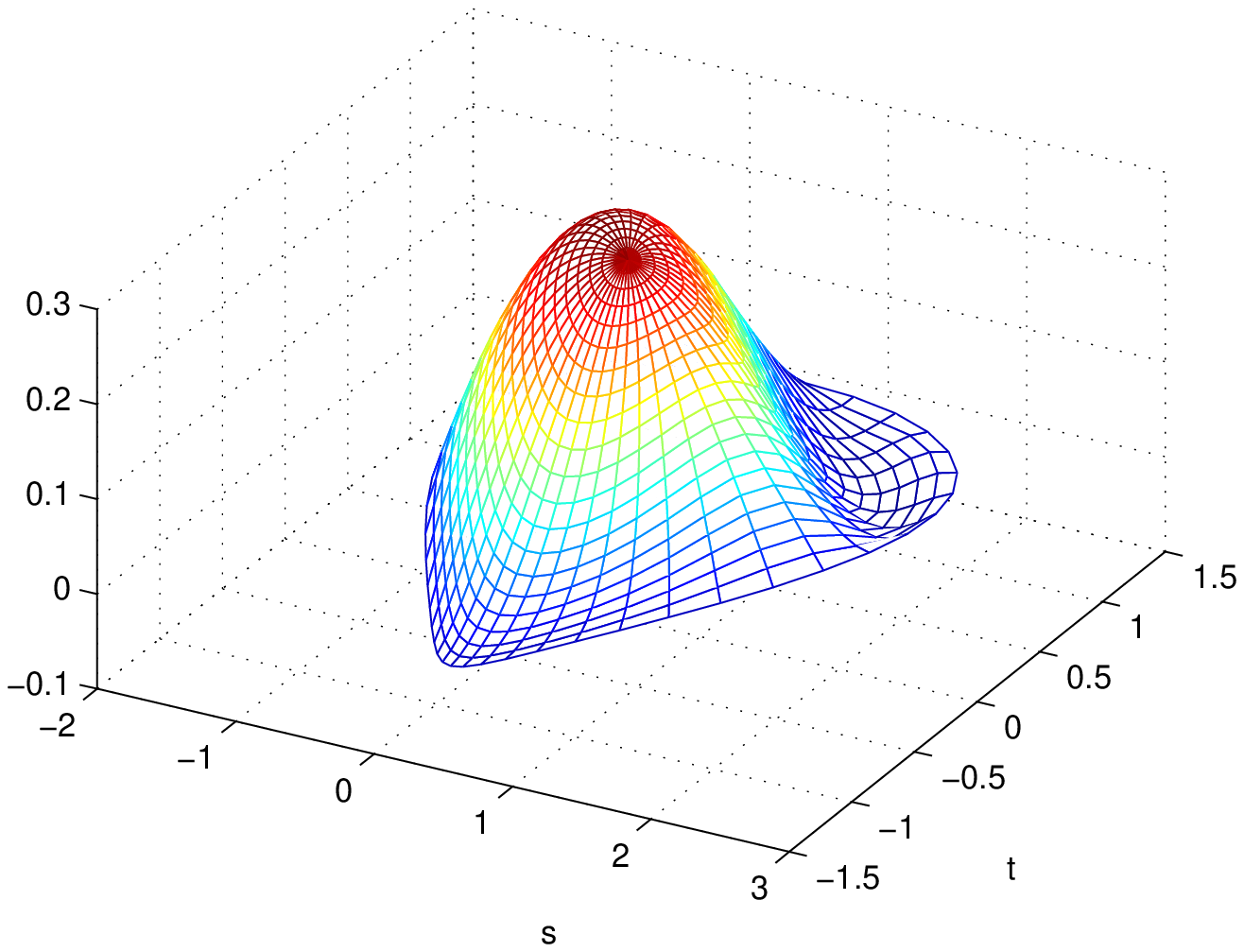}%
\\
The solution $u$%
\end{center}}}
&
{\parbox[b]{2.0003in}{\begin{center}
\includegraphics[
height=2.0003in,
width=2.0003in
]%
{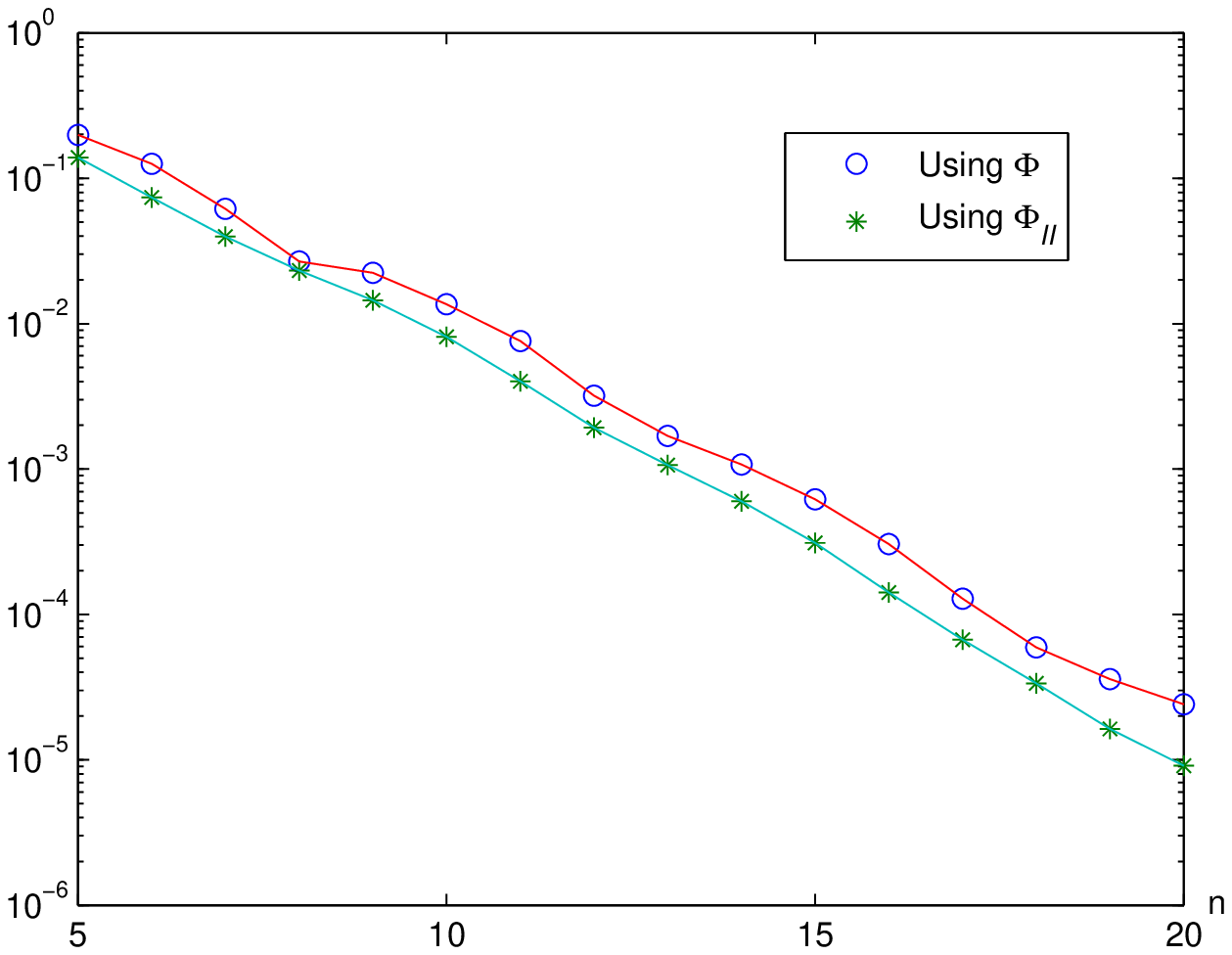}%
\\
The maximum error
\end{center}}}
\end{tabular}
\caption{The solution u to (\ref{e120}) with right side (\ref{e135}) and its
error}\label{PlanarSoln}%
\end{figure}%

As a second planar example we consider the stationary Fisher equation where
the function $f$ in (\ref{e120}) is given by
\[
f(s,t,u)=100u(1-u),\quad\quad(s,t)\in\Omega.
\]
Fisher's equation is used to model the spreading of biological populations and
from $f$ we see that $u=0$ and $u=1$ are stationary points for the time
dependent equation on an unbounded domain; see \cite[Chap. 17]{Kot}. The
original Fisher equation does not contain the term $100$, but for small
domains the Fisher equation might have no nontrivial solution and the factor
100 corresponds to a scaling by a factor 10 to guarantee the existence of a
nontrivial solution on the domain $\Omega$. The domain $\Omega$ is the
interior of the curve
\begin{equation}
\varphi(t)=(3+\cos(t)+2\sin(t))\left(  \cos t,\sin t\right) \label{e201}%
\end{equation}
We studied this domain in earlier papers (see \cite{ah2011}) where we called
this domain a `Limacon domain'. In the article \cite{ah2011} we also describe
how we use equation (\ref{e201}) to create a domain mapping $\Phi
:\overline{\mathbb{B}}^{2}\rightarrow\overline{\Omega}$ by two dimensional
interpolation. Similar to the previous example we calculate the numerical
solutions $u_{n}$ for $n=1,\ldots,40$, where we use the coefficients of
$u_{n-1}$ as a starting value $u_{n}^{(0)}$ for $n=2,\ldots,40$ and for
$u_{1}^{(0)}$ we use coefficients which are non zero (all equal to 10), so the
iteration of \texttt{fsolve} does not converge to the trivial solution. As a
reference solution we calculated $u_{45}$; see Figure \ref{FisherRef}.%

\begin{figure}[tbp] \centering
\begin{tabular}
[c]{ll}%
{\parbox[b]{2.0003in}{\begin{center}
\includegraphics[
height=2.0003in,
width=2.0003in
]%
{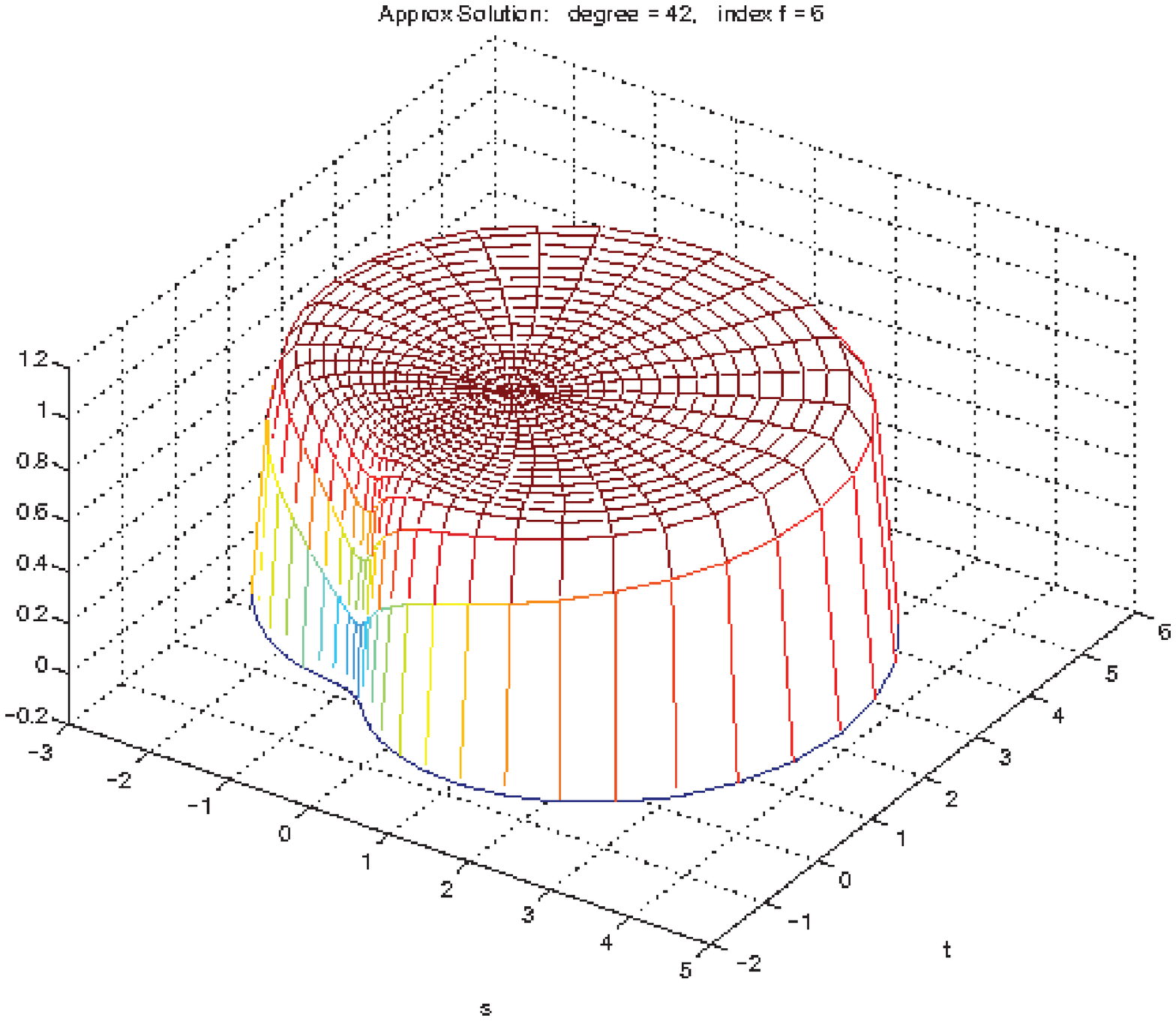}%
\\
Solution for Fisher's equation
\end{center}}}
&
{\parbox[b]{2.0003in}{\begin{center}
\includegraphics[
height=2.0003in,
width=2.0003in
]%
{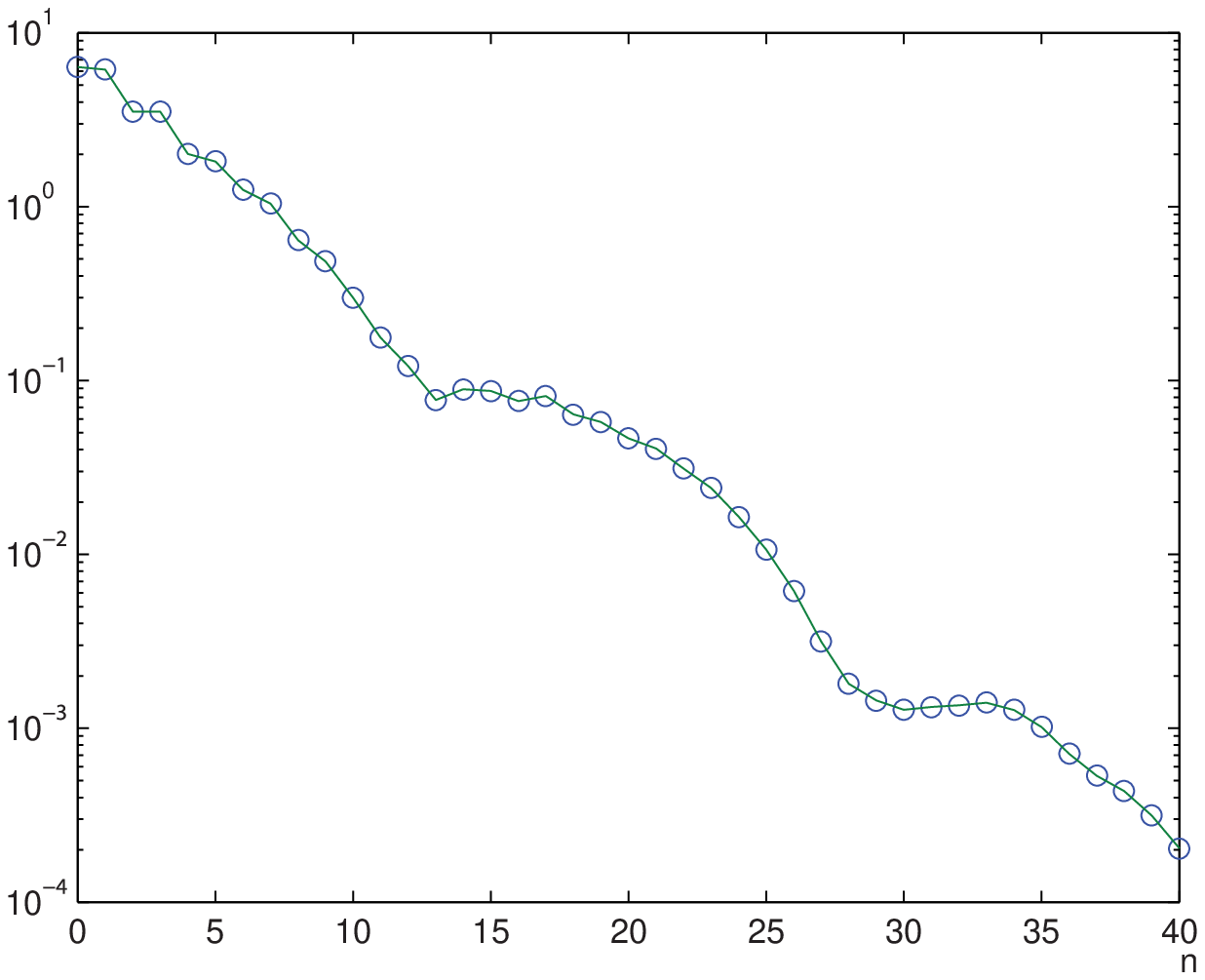}%
\\
Error for Fisher's equation
\end{center}}}
\end{tabular}
\caption{The reference solution and maximum error for Fisher's equation}\label{FisherRef}%
\end{figure}%

The shape of the solution is very much like we expect it, the function is
close to $1$ inside the domain $\Omega$ and drops off very steeply to the
boundary value $0$. By looking at the reference solution in Figure
\ref{FisherRef} we also see that the function will be harder to approximate by
polynomials than the function in the previous example, because of the sharp
drop off. This becomes clear when we look at the convergence, also shown in
Figure \ref{FisherRef}. The final error is in the range of $10^{-3}$%
--$10^{-4}$ with a polynomial degree of 40, so the error is in the same range
as in the previous example where we only used polynomials up to degree 20 for
the approximation. Still the graph suggests that the convergence is
exponential as predicted by (\ref{en13}) for the $L^{2}$ norm.

\subsection{A three dimensional example}

In the following we present a three dimensional example. We use the mapping
$\Phi:\overline{\mathbb{B}}^{3}\rightarrow\overline{\Omega}$, $\left(
s,t,v\right)  =\Phi\left(  x,y,z\right)  $, defined by
\begin{equation}%
\begin{array}
[c]{l}%
s=x-y+ax^{2}\\
t=x+y\\
v=2z+bz^{2}%
\end{array}
\label{e202}%
\end{equation}
where $a=b=0.5$. We have used this mapping in a previous article, see
\cite{ach2008}, where one finds plots of the surface $\partial\Omega$. On
$\Omega$ we solve
\begin{equation}%
\begin{tabular}
[c]{cc}%
$-\Delta u\left(  s,t,v\right)  =f\left(  s,t,v,u\left(  s,t,v\right)
\right)  ,$ & $\quad\left(  s,t,v\right)  \in\Omega$\\
$u\left(  s,t,v\right)  =0,$ & $\quad\left(  s,t,v\right)  \in\partial\Omega$%
\end{tabular}
\label{e203}%
\end{equation}
where $f$ is defined by
\[
f(s,t,v,u)=\frac{\cos(6x+y+z)}{1+u^{2}},\quad\quad(s,t,v)\in\Omega
\]
We calculated approximate solutions $u_{1},\ldots,u_{20}$ and used $u_{25}$ as
a reference solution. In Figure \ref{Ex3D} we see the convergence in the
maximum norm on a grid in $\overline{\mathbb{B}}^{3}$. As in our previous
examples the graph suggests that we have exponential convergence.%

\begin{figure}[ptb]%
\centering
\includegraphics[
height=3in,
width=3.9998in
]%
{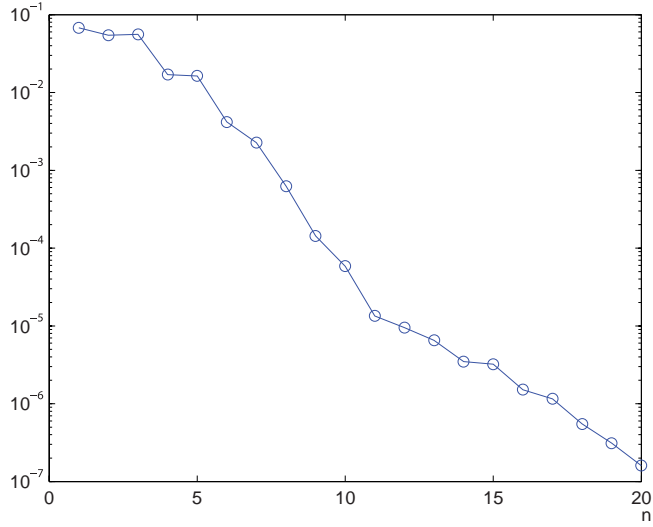}%
\caption{For the problem (\ref{e203}), the convergence of the \ errors
$\left\Vert u-u_{n}\right\Vert _{\infty}$.}%
\label{Ex3D}%
\end{figure}

In our final Figure \ref{Ex3Dplots} we show the graph of the reference
solution $u_{25}$ on $\overline{\mathbb{B}}^{3}\cap P_{\nu}$ where $P_{\nu} $
is a plane in $\mathbb{R}^{3}$ normal to the vector $\nu$. We have used
several normal vectors $\nu_{1}=(0,0,1)^{T}$, so $P_{\nu_{1}}$ is the
$xy$--plane, $\nu_{2}=(0,0,1)^{T}$, so $P_{\nu_{2}}$ is the $xz$--plane,
$\nu_{3}=(1,0,0)^{T}$, so $P_{\nu_{3}}$ is the $yz$--plane, and $\nu
_{4}=(1,1,1)^{T}$, so $P_{\nu_{4}}$ is a diagonal plane. Figure
\ref{Ex3Dplots} shows that the solution reflects the periodic character of the
nonlinearity $f$. In the yz--plane the oscillation of $f$ is much slower which
is also visible in the plot along the $yz$--plane.%

\begin{figure}[tbp] \centering
\begin{tabular}
[c]{ll}%
{\parbox[b]{2.0003in}{\begin{center}
\includegraphics[
height=2.0003in,
width=2.0003in
]%
{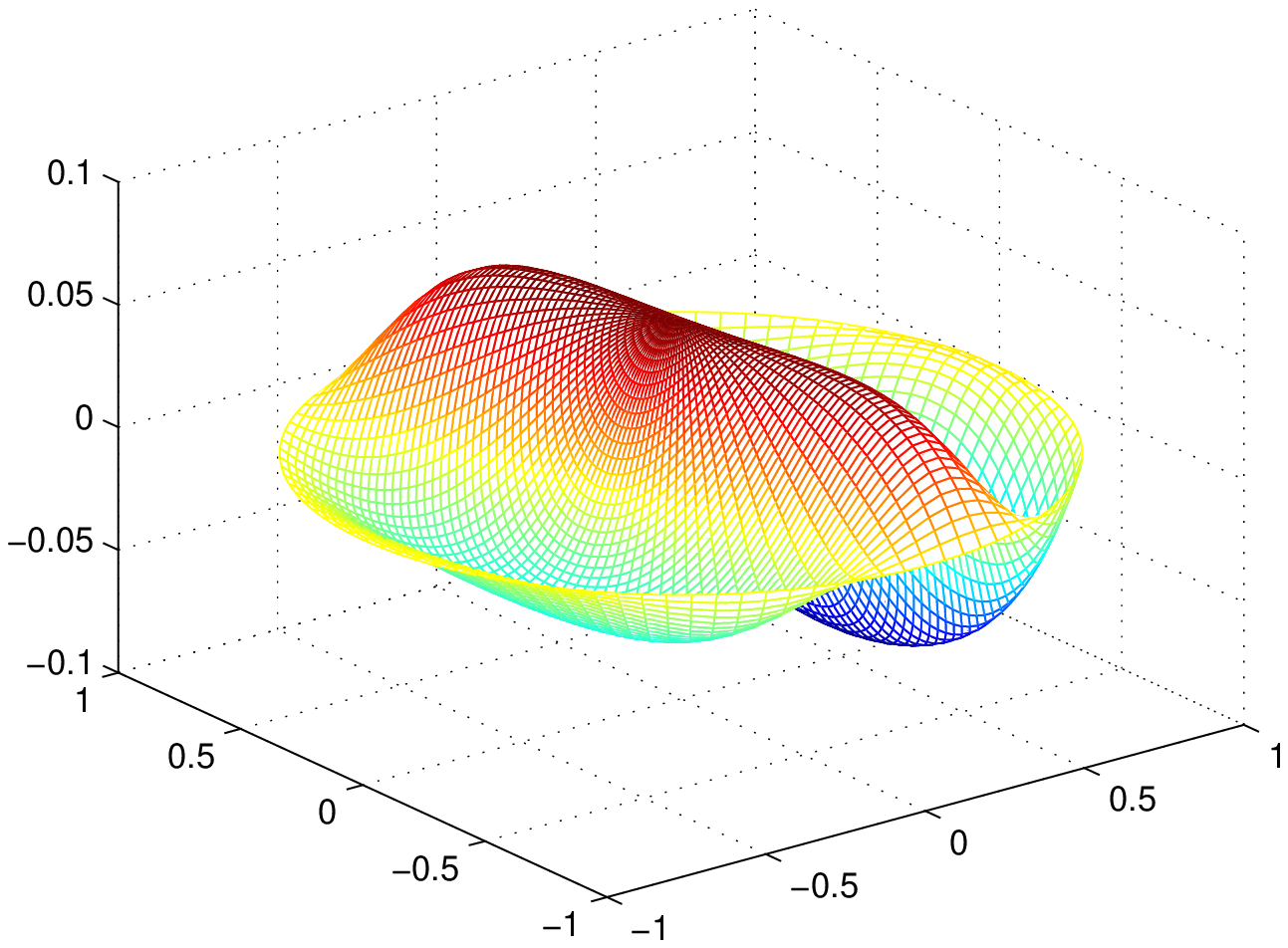}%
\\
$\nu_{1}=(0,0,1),$ $P_{1}$ is $xy-plane$%
\end{center}}}
&
{\parbox[b]{2.0003in}{\begin{center}
\includegraphics[
height=2.0003in,
width=2.0003in
]%
{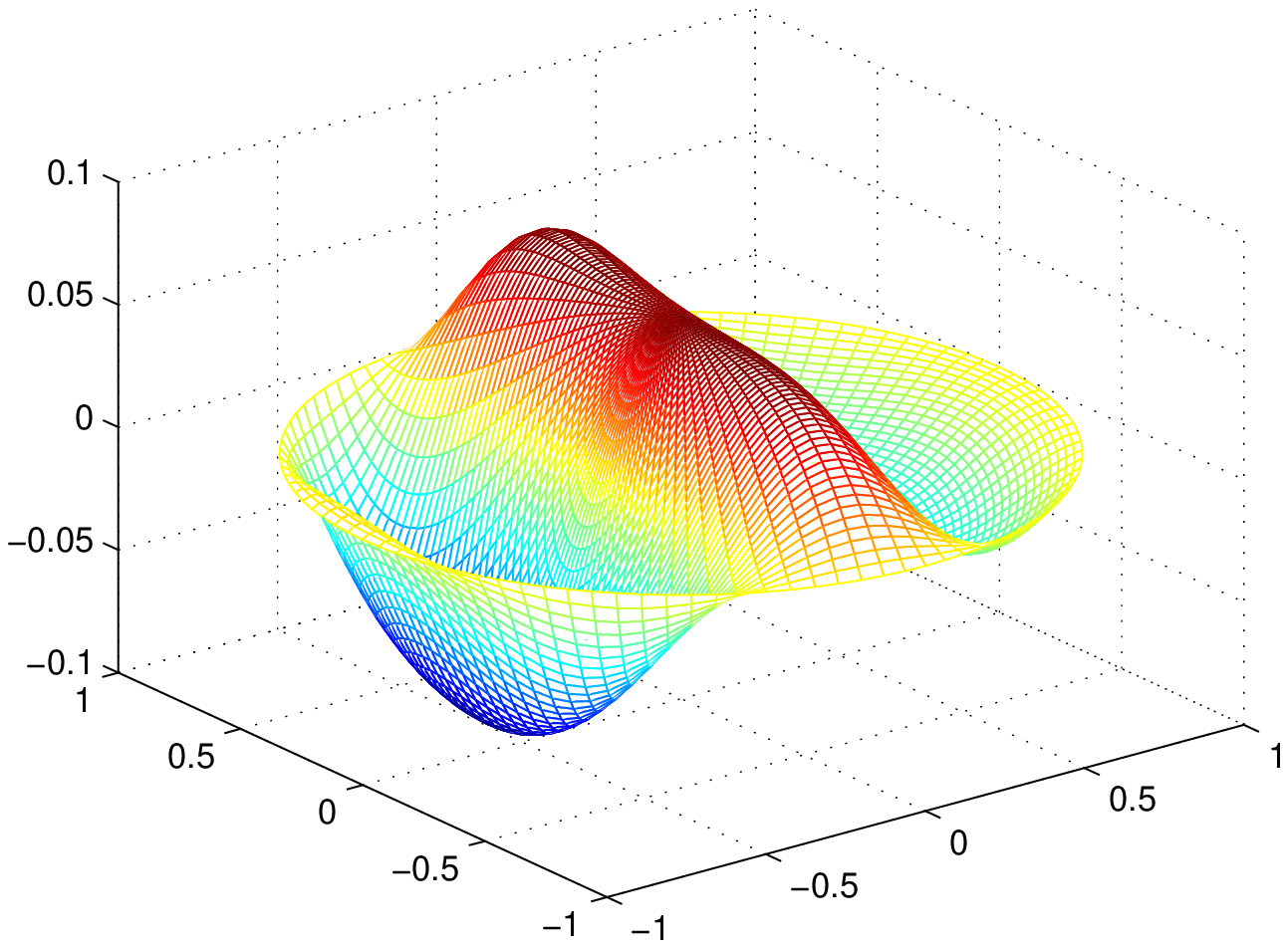}%
\\
$\nu_{2}=(0,1,0),$ $P_{2}$ is $xz-plane $%
\end{center}}}
\\%
{\parbox[b]{2.0003in}{\begin{center}
\includegraphics[
height=2.0003in,
width=2.0003in
]%
{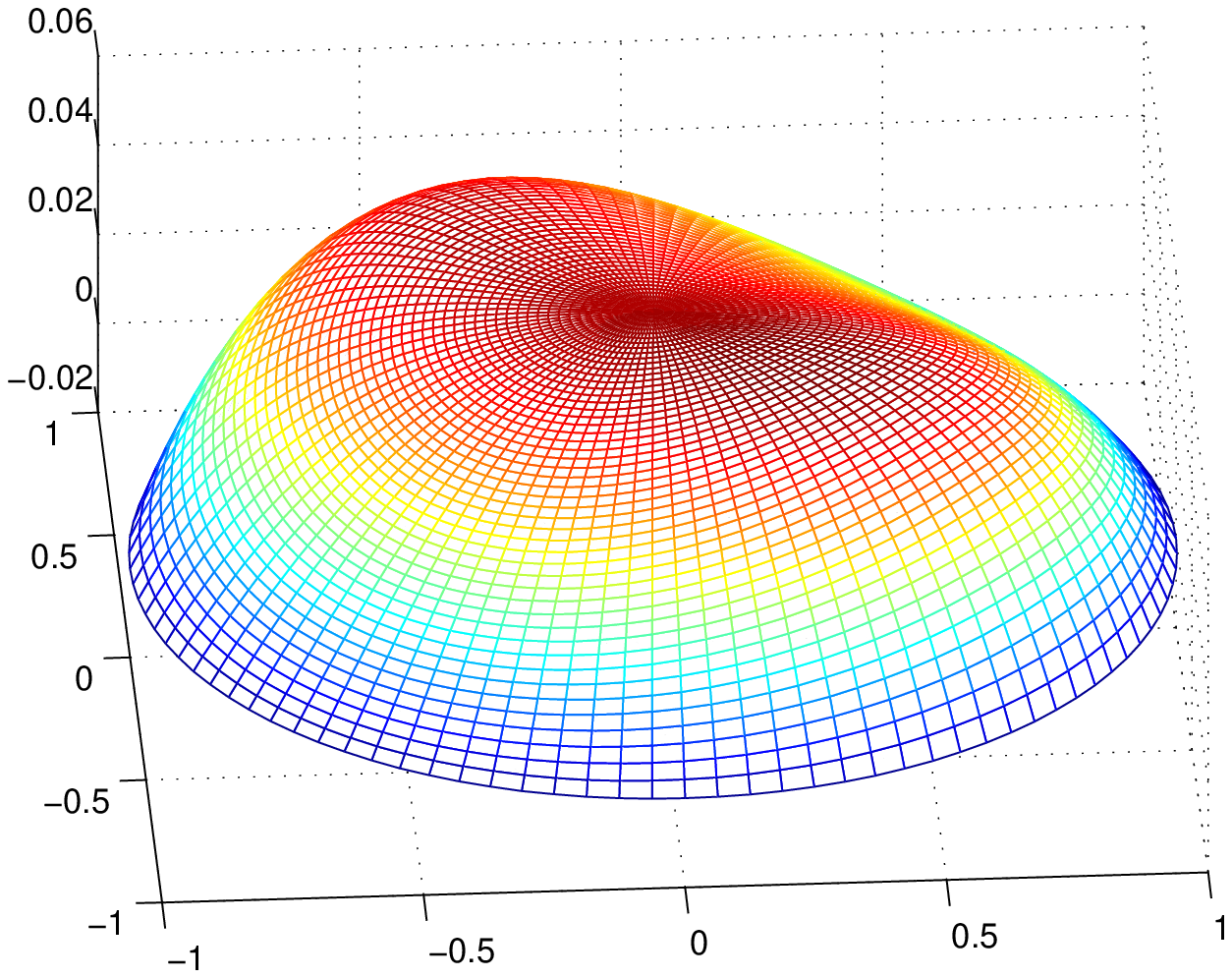}%
\\
$\nu_{3}=(1,0,0),$ $P_{3}$ is $yz-plane$%
\end{center}}}
&
{\parbox[b]{2.0003in}{\begin{center}
\includegraphics[
height=2.0003in,
width=2.0003in
]%
{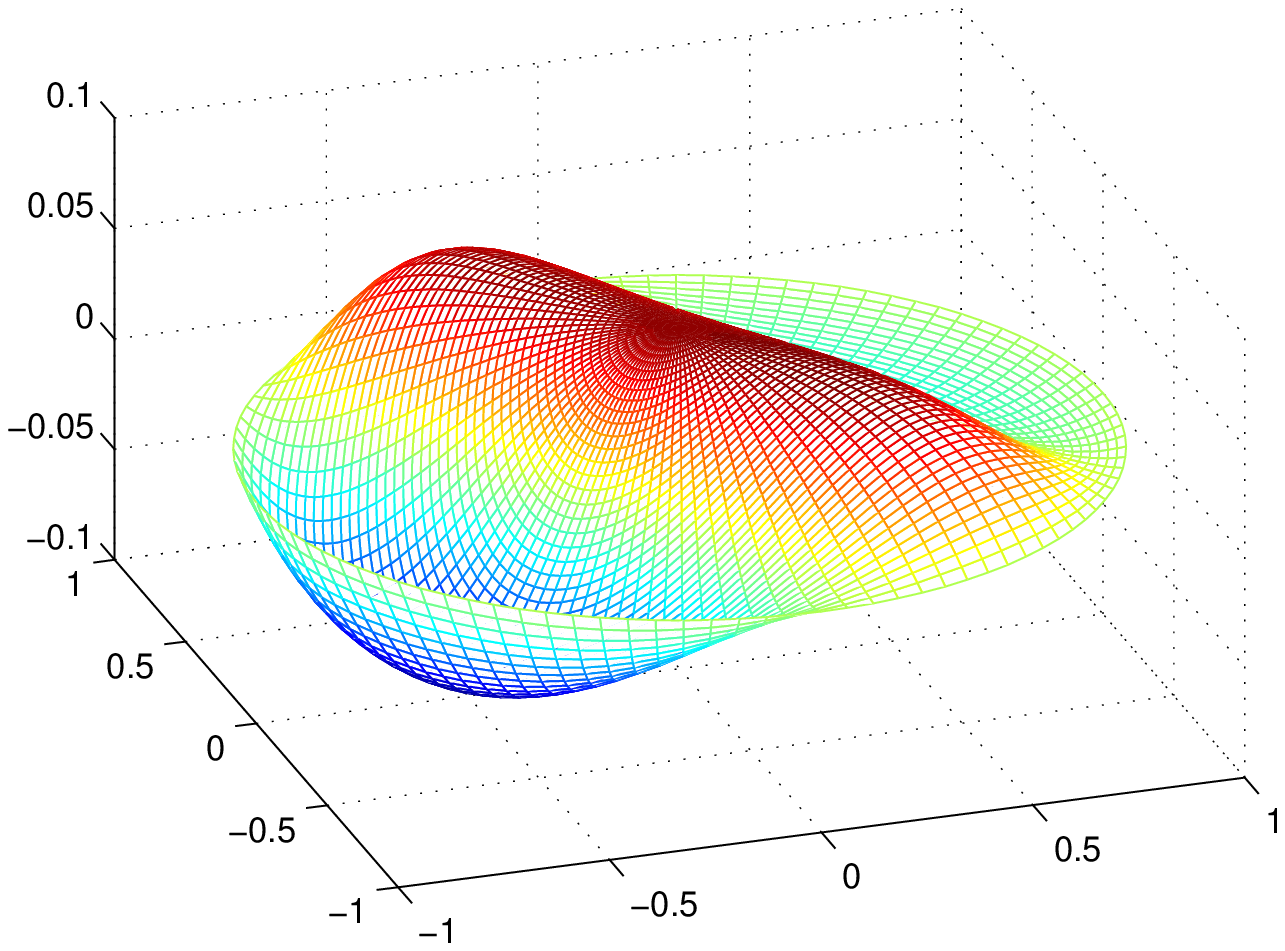}%
\\
$\nu_{4}=(1,1,1),$ $P_{4}$ is diagonal
\end{center}}}
\end{tabular}
\caption{The solution $\widetilde{u}\left( x,y,z\right) $ over $P\cap
\mathbb{B}^{3}$ with $P$ a plane passing through the origin and orthogonal to $\nu$}\label{Ex3Dplots}%
\end{figure}%

\section{Neumann boundary value problem\label{Nmn}}

Consider the boundary value problem%
\begin{equation}
-\Delta u\left(  s\right)  +\gamma\left(  s\right)  u\left(  s\right)
=f\left(  s,u(s)\right)  ,\quad\quad s\in\Omega,\label{e150}%
\end{equation}%
\begin{equation}
\frac{\partial u\left(  s\right)  }{\partial n_{s}}=0,\quad\quad
s\in\mathbb{\partial}\Omega.\label{e151}%
\end{equation}
with $n_{s}$ the exterior unit normal to $\partial\Omega$ at the boundary
point $s$. Later we discuss an extension to a nonzero normal derivative over
$\partial\Omega$. A necessary condition for the unknown function $u^{\ast} $
to be a solution of (\ref{e150})-(\ref{e151}) is that it satisfy%
\begin{equation}
\int_{\Omega}f\left(  s,u^{\ast}\left(  s\right)  \right)  \,ds=\int_{\Omega
}\gamma\left(  s\right)  u^{\ast}\left(  s\right)  \,ds.\label{e151a}%
\end{equation}
With our assumption that (\ref{e150})-(\ref{e151}) has a locally unique
solution $u^{\ast}$, (\ref{e151a}) is satisfied.

Proceed in analogy with the earlier treatment of the Dirichlet problem. Use
integration by parts to show that for arbitrary functions $u\in H^{2}\left(
\Omega\right)  ,$ $v\in H^{1}\left(  \Omega\right)  $,%
\begin{equation}%
\begin{array}
[c]{r}%
{\displaystyle\int_{\Omega}}
v(s)\left[  -\Delta u(s)+\gamma(s)u\right]  \,ds=%
{\displaystyle\int_{\Omega}}
\left[  \triangledown u(s)\cdot\triangledown v(s)+\gamma(s)u(s)v(s)\right]
\,ds\quad\medskip\\
-%
{\displaystyle\int_{\partial\Omega}}
v\left(  s\right)  \dfrac{\partial u(s)}{\partial n_{s}}\,ds.
\end{array}
\label{e151b}%
\end{equation}
Introduce the bilinear functional%
\[
\mathcal{A}\left(  v_{1},v_{2}\right)  =\int_{\Omega}\left[  \triangledown
v_{1}(s)\cdot\triangledown v_{2}(s)+\gamma(s)v_{1}(s)v_{2}(s)\right]  \,ds.
\]
The variational form of the Neumann problem (\ref{e150})-(\ref{e151}) is as
follows: find $u\in H^{1}\left(  \Omega\right)  $ such that%
\begin{equation}
\mathcal{A}\left(  u,v\right)  =\left(  \mathcal{F}\left(  u\right)
,v\right)  ,\quad\quad\forall v\in H^{1}\left(  \Omega\right) \label{e152}%
\end{equation}
with, as before, the operator $\mathcal{F}$ defined by%
\[
\left(  \mathcal{F}\left(  u\right)  \right)  \left(  s\right)  =f(s,u\left(
s\right)  ).
\]
The theory for (\ref{e152}) is essentially the same as for the Dirichlet
problem in its reformulation (\ref{en3}).

Because of changes that take place in the normal derivative under the
transformation $s=\Phi\left(  x\right)  $, we modify the construction of the
numerical method. In the actual implementation, however, it will mirror that
for the Dirichlet problem. For the approximating space, let
\[
\mathcal{X}_{n}=\left\{  q\mid q\circ\Phi=p\text{ for some }p\in\Pi_{n}%
^{d}\right\}  .
\]
For the numerical method, we seek $u_{n}^{\ast}\in\mathcal{X}_{n}$ for which%
\begin{equation}
\mathcal{A}\left(  u_{n}^{\ast},v\right)  =\left(  \mathcal{F}\left(
u_{n}^{\ast}\right)  ,v\right)  ,\quad\quad\forall v\in\mathcal{X}%
_{n}.\label{e153}%
\end{equation}
A similar approach was used in \cite{ahc2009} for the linear Neumann problem.

To carry out a convergence analysis for (\ref{e153}), it is necessary to
compare convergence of approximants in $\mathcal{X}_{n}$ to that of
approximants from $\Pi_{n}^{d}.$ For simplicity in notation, we assume
$\Phi\in C^{\infty}\left(  \overline{\mathbb{B}}^{d}\right)  $. Begin by
referring to Lemma \ref{transform_lemma} and its discussion in
\S \ref{section_transform}, linking differentiability in $H^{m}\left(
\Omega\right)  $ and $H^{m}\left(  \mathbb{B}^{d}\right)  $. In particular,
for $m\geq0$,
\begin{equation}
c_{1,m}\left\Vert v\right\Vert _{H^{m}\left(  \Omega\right)  }\leq\left\Vert
\widetilde{v}\right\Vert _{H^{m}\left(  \mathbb{B}^{d}\right)  }\leq
c_{2,m}\left\Vert v\right\Vert _{H^{m}\left(  \Omega\right)  },\quad\quad v\in
H^{m}\left(  \Omega\right)  ,\label{e154}%
\end{equation}
with $\widetilde{v}=v\circ\Phi$, with constants $c_{1,m},c_{2,m}>0$.

Also recall Theorem \ref{Thm1a} concerning approximation of functions
$\widetilde{v}\in H^{r}\left(  \mathbb{B}^{d}\right)  $ and link this to
approximation of functions $v\in H^{r}\left(  \Omega\right)  $.

\begin{lemma}
Assume $\Phi\in C^{\infty}\left(  \overline{\mathbb{B}}^{d}\right)  $. Assume
$v\in H^{r}\left(  \Omega\right)  $ for some $r\geq2$. Then there exist a
sequence $q_{n}\in\mathcal{X}_{n}$, $n\geq1$, for which
\begin{equation}
\left\Vert v-q_{n}\right\Vert _{H^{1}\left(  \Omega\right)  }\leq
\varepsilon_{n,r}\left\Vert v\right\Vert _{H^{r}\left(  \Omega\right)  }%
,\quad\quad n\geq1.\label{e155}%
\end{equation}
\textit{The sequence }$\varepsilon_{n,r}=\mathcal{O}\left(  n^{-r+1}\right)
$\textit{\ and is independent of }$v$\textit{.}
\end{lemma}

\begin{proof}
Begin by applying Theorem \ref{Thm1a} to the function $\widetilde{v}\left(
x\right)  =v\left(  \Phi\left(  x\right)  \right)  $. Then there is a sequence
of polynomials $p_{n}\in\Pi_{n}^{d}$ for which%
\[
\left\Vert \widetilde{v}-p_{n}\right\Vert _{H^{1}\left(  \mathbb{B}%
^{d}\right)  }\leq\varepsilon_{n,r}\left\Vert \widetilde{v}\right\Vert
_{H^{r}\left(  \mathbb{B}^{d}\right)  },\quad\quad n\geq1.
\]
Let $q_{n}=p_{n}\circ\Phi^{-1}$. The result then follows by applying
(\ref{e154}). $\left.  {}\right.  $\hfill\medskip
\end{proof}

The theoretical convergence analysis now follows exactly that given earlier
for the Dirichlet problem. Again we use the construction from \cite[\S 4(a)]%
{Osborn1975}, but now use the integral operator $\mathcal{T}$ arising from the
zero Neumann boundary condition. As with the Dirichlet problem, it is
necessary to have $\mathcal{A}$ be strongly elliptic, and for that reason and
without any loss of generality, assume
\[
\min_{s\in\overline{\Omega}}\gamma\left(  s\right)  >0.
\]
The solution of (\ref{e152}) can be written as $u=\mathcal{TF}\left(
u\right)  $ with $\mathcal{T}:L_{2}\left(  \mathbb{B}^{d}\right)  \rightarrow
H^{2}\left(  \mathbb{B}^{d}\right)  $ and bounded. Use Theorem \ref{Thm1a} in
place of Theorem \ref{Thm1b} for polynomial approximation error, as in the
derivation of (\ref{en11a}). \ Theorems \ref{thm2} and \ref{ThmDir}, along
with Corollary \ref{CorDir} are valid for the spectral method for the Neumann
problem (\ref{e150})-(\ref{e151}).

\subsection{Implementation}

As in \S \ref{implement}, we look for a solution to (\ref{e152}) by looking
for
\begin{equation}
u_{n}\left(  s\right)  =\sum_{\ell=1}^{N_{n}}\alpha_{\ell}\psi_{\ell}\left(
s\right) \label{e156}%
\end{equation}
with \ $\left\{  \psi_{\ell}:1\leq j\leq N_{n}\right\}  $ a basis for
$\mathcal{X}_{n}$. The system associated with (\ref{e152}) that is to be
solved is%
\begin{equation}%
\begin{array}
[c]{l}%
{\displaystyle\sum\limits_{\ell=1}^{N_{n}}}
\alpha_{\ell}%
{\displaystyle\int_{\Omega}}
\left[
{\displaystyle\sum\limits_{i,j=1}^{d}}
a_{i,j}(s)\dfrac{\partial\psi_{\ell}(s)}{\partial s_{i}}\dfrac{\partial
\psi_{k}(s)}{\partial s_{j}}+\gamma\left(  s\right)  \psi_{\ell}\left(
s\right)  \psi_{k}\left(  s\right)  \right]  \,ds\quad\smallskip\\
\quad\quad=%
{\displaystyle\int_{\Omega}}
f\left(  s,%
{\displaystyle\sum\limits_{\ell=1}^{N_{n}}}
\alpha_{\ell}\psi_{\ell}\left(  s\right)  \right)  \psi_{k}(s)\,ds,\quad\quad
k=1,\dots,N_{n}.
\end{array}
\label{e157}%
\end{equation}

For such a basis $\left\{  \psi_{\ell}\right\}  $, we begin with an
orthonormal basis for $\Pi_{n}$, say $\{ \varphi_{j}:1\leq j\leq N_{n}\}$, and
then define%
\[
\psi_{\ell}\left(  s\right)  =\varphi_{\ell}\left(  x\right)  \quad
\text{\emph{with\quad}}s=\Phi\left(  x\right)  ,\quad\quad1\leq\ell\leq N.
\]
The function $\widetilde{u}_{n}\left(  x\right)  \equiv u_{n}\left(
\Phi\left(  x\right)  \right)  $, $x\in\mathbb{B}^{d},$ is to be the
equivalent solution considered over $\mathbb{B}^{d}.$ Using the transformation
of variables $s=\Phi\left(  x\right)  $ in the system (\ref{e157}), the
coefficients $\left\{  \alpha_{\ell}|\ell=1,2,\dots,N_{n}\right\}  $ are the
solutions of
\begin{equation}%
\begin{array}
[c]{r}%
{\displaystyle\sum\limits_{k=1}^{N_{n}}}
\alpha_{k}%
{\displaystyle\int_{\mathbb{B}^{d}}}
\left[
{\displaystyle\sum\limits_{i,j=1}^{d}}
\widetilde{a}_{i,j}(x)\dfrac{\partial\varphi_{k}(x)}{\partial x_{j}}%
\dfrac{\partial\varphi_{\ell}(x)}{\partial x_{i}}+\gamma(\Phi\left(  x\right)
)\varphi_{k}(x)\varphi_{\ell}(x)\right]  \det J\left(  x\right)
\,dx\medskip\\
=%
{\displaystyle\int_{\mathbb{B}^{d}}}
f\left(  x,%
{\displaystyle\sum\limits_{k=1}^{N_{n}}}
\alpha_{k}\varphi_{k}\left(  x\right)  \right)  \varphi_{\ell}\left(
x\right)  \det J\left(  x\right)  \,dx,\quad\quad\ell=1,\dots,N_{n}%
\end{array}
\label{e158}%
\end{equation}
For the equation (\ref{e150}) the matrix $A\left(  s\right)  $ is the
identity, and therefore from (\ref{en8ac}),
\[
\widetilde{A}\left(  x\right)  =J\left(  x\right)  ^{-1}J\left(  x\right)
^{-\text{T}}.
\]
The system (\ref{e158}) is much the same as (\ref{e78}) for the Dirichlet
problem, differing only by the basis functions being used for the solution
$\widetilde{u}_{n}$. We use the same numerical integration as before, and also
the same orthonormal basis for $\Pi_{n}^{d}$.

\subsection{Numerical example}

Consider the problem%
\begin{equation}%
\begin{tabular}
[c]{cc}%
$-\Delta u\left(  s,t\right)  +u\left(  s,t\right)  =f\left(  s,t,u\left(
s,t\right)  \right)  ,$ & $\quad\left(  s,t\right)  \in\Omega$\\
$\dfrac{\partial u\left(  s\right)  }{\partial n_{s}}=0,$ & $\quad\left(
s,t\right)  \in\partial\Omega$%
\end{tabular}
\label{e160}%
\end{equation}
with $\Omega$ the elliptical region%
\[
\left(  \frac{s}{a}\right)  ^{2}+\left(  \frac{t}{b}\right)  ^{2}\leq1.
\]
The mapping of $\mathbb{B}^{2}$ onto $\Omega$ is simply%
\[
\Phi\left(  x,y\right)  =\left(  ax,by\right)  ,\quad\quad\left(  x,y\right)
\in\overline{\mathbb{B}}^{2}.
\]
As before, note the change in notation, from $s\in\mathbb{R}^{2}$ to $\left(
s,t\right)  \in\mathbb{R}^{2}$.

The right side $f$ is given by
\begin{equation}
f\left(  s,t,u\right)  =-e^{u}+f_{1}\left(  s,t\right) \label{e161}%
\end{equation}
with the function $f_{1}$ determined from the given true solution and the
equation (\ref{e160}) to define $f\left(  s,t,u\right)  $. In our case,%
\begin{equation}
u\left(  s,t\right)  =\left(  1-\left(  \frac{s}{a}\right)  ^{2}-\left(
\frac{t}{b}\right)  ^{2}\right)  ^{2}\cos\left(  2s+t^{2}\right)
.\label{e162}%
\end{equation}
Easily this has a normal derivative of zero over the boundary of $\Omega$.%

\begin{figure}[tbp] \centering
\begin{tabular}
[c]{ll}%
{\parbox[b]{2.0003in}{\begin{center}
\includegraphics[
height=2.0003in,
width=2.0003in
]%
{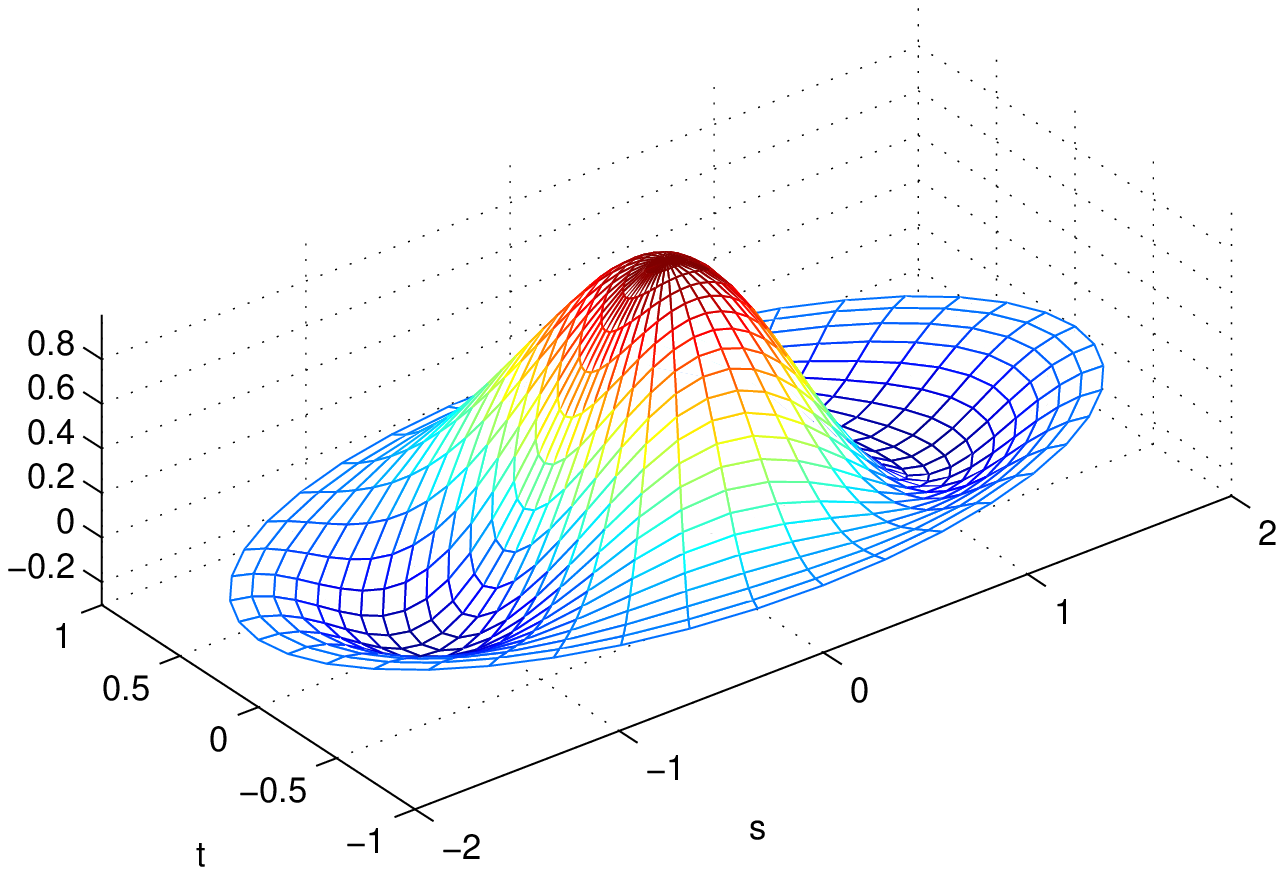}%
\\
Solution (\ref{e162})
\end{center}}}
&
\raisebox{-0.0208in}{\parbox[b]{2.0003in}{\begin{center}
\includegraphics[
height=2.0003in,
width=2.0003in
]%
{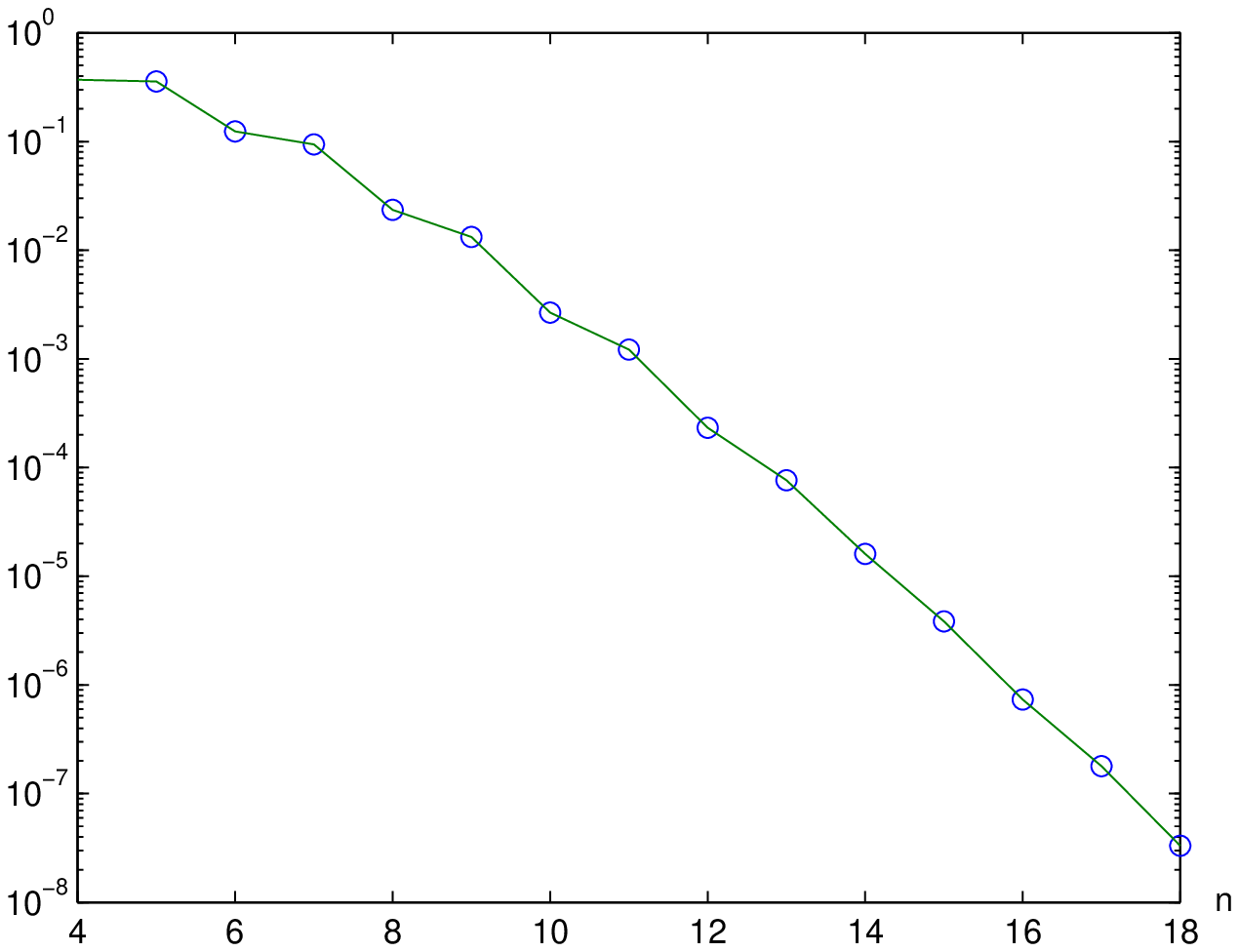}%
\\
Maximum error
\end{center}}}
\end{tabular}
\caption{The solution u to (\ref{e160}) with right side (\ref{e161}) and true
solution (\ref{e162})}\label{nmnsoln}%
\end{figure}%

The nonlinear system (\ref{e158}) was solved using \texttt{fsolve} from
\textsc{Matlab,} as earlier in \S \ref{NumExam}. Our region $\Omega$ uses
$\left(  a,b\right)  =\left(  2,1\right)  $. Figure \ref{nmnsoln} contains the
approximate solution for $n=18$ and also shows the maximum error over
$\overline{\Omega}$. Again, the convergence appears to be exponential.

\subsection{Handling a nonzero Neumann condition}

Consider the problem
\begin{equation}
-\Delta u\left(  s\right)  +\gamma\left(  s\right)  u\left(  s\right)
=f\left(  s,u(s)\right)  ,\quad\quad s\in\Omega,\label{e170}%
\end{equation}%
\begin{equation}
\frac{\partial u\left(  s\right)  }{\partial n_{s}}=g(s),\quad\quad
s\in\mathbb{\partial}\Omega\label{e171}%
\end{equation}
with a nonzero Neumann boundary condition. Let $u^{\ast}\left(  s\right)  $
denote the solution we are seeking. A necessary condition for solvability of
(\ref{e170})-(\ref{e171}) is that
\begin{equation}
\int_{\Omega}f\left(  s,u^{\ast}\left(  s\right)  \right)  \,ds=\int_{\Omega
}\gamma\left(  s\right)  u^{\ast}\left(  s\right)  \,ds-\int_{\partial\Omega
}g\left(  s\right)  \,ds.\label{e171a}%
\end{equation}
There are at least two approaches to extending our spectral method to solve
this problem.

First, consider the problem%
\begin{equation}
-\Delta v\left(  s\right)  =c_{0},\quad\quad s\in\Omega,\label{e172}%
\end{equation}%
\begin{equation}
\frac{\partial v\left(  s\right)  }{\partial n_{s}}=g(s),\quad\quad
s\in\mathbb{\partial}\Omega,\label{e173}%
\end{equation}
with $c_{0}$ a constant. From (\ref{e171a}), solvability of (\ref{e172}%
)-(\ref{e173}) requires
\begin{equation}
\int_{\Omega}c_{0}\,ds=-\int_{\partial\Omega}g\left(  s\right)
\,ds,\label{e173a}%
\end{equation}
is satisfied. To do so, choose%
\[
c_{0}=\frac{-1}{\operatorname*{Vol}\left(  \Omega\right)  }\int_{\partial
\Omega}g\left(  s\right)  \,ds.
\]
A solution $v^{\ast}\left(  s\right)  $ exists, although it is not unique. The
solution of (\ref{e172})-(\ref{e173}) can be approximated using the method
given in \cite{ahc2009}. Then introduce%
\[
w=u-v^{\ast}.
\]
Substituting into (\ref{e170})-(\ref{e171}), the new unknown function
$w^{\ast}$ satisfies%
\begin{equation}
-\Delta w\left(  s\right)  +\gamma\left(  s\right)  w\left(  s\right)
=f\left(  s,w(s)+v^{\ast}\left(  s\right)  \right)  -\gamma\left(  s\right)
v^{\ast}\left(  s\right)  -c_{0},\quad\quad s\in\Omega,\label{e174}%
\end{equation}%
\begin{equation}
\frac{\partial w\left(  s\right)  }{\partial n_{s}}=0,\quad\quad
s\in\mathbb{\partial}\Omega.\label{e175}%
\end{equation}
The methods of this section can be used to approximate $w^{\ast}$; and then
use $u^{\ast}=w^{\ast}+v^{\ast}.$

A second approach is to use (\ref{e151b}) to reformulate (\ref{e170}%
)-(\ref{e171}) as the problem of finding $u=u^{\ast}$ for which%
\begin{equation}
\mathcal{A}\left(  u,v\right)  =\left(  \mathcal{F}\left(  u\right)
,v\right)  +\ell\left(  v\right)  ,\quad\quad\forall v\in H^{1}\left(
\Omega\right) \label{e176}%
\end{equation}
with
\[
\ell\left(  v\right)  =\int_{\partial\Omega}v\left(  s\right)  g\left(
s\right)  \,ds.
\]
Thus we seek
\[
u_{n}\left(  s\right)  =\sum_{\ell=1}^{N_{n}}\alpha_{\ell}\psi_{\ell}\left(
s\right)
\]
for which%
\begin{equation}
\mathcal{A}\left(  u_{n},v\right)  =\left(  \mathcal{F}\left(  u\right)
,v\right)  +\ell\left(  v\right)  ,\quad\quad\forall v\in\mathcal{X}%
_{n}.\label{e177}%
\end{equation}

The first approach, that of (\ref{e170})-(\ref{e175}), is usable, and the
convergence analysis follows from combining this paper's analysis with that of
\cite{ahc2009}. Unfortunately, we do not have a convergence analysis \ for
this second approach, that of (\ref{e176})-(\ref{e177}), as the Green's
function approach of this paper does not seem to extend to it.

\end{document}